\documentclass{article}

\usepackage[a4paper,inner=2.0cm,top=2.3cm,outer=2.0cm,bottom=2cm]{geometry}

\usepackage{enumitem}

\usepackage{amsmath}
\usepackage{amsfonts}
\usepackage{amssymb}
\usepackage{amsthm}
\usepackage{mathtools}
\usepackage{hyperref}
\usepackage{cleveref}
\usepackage[numbers]{natbib}
\crefrangeformat{equation}{eqs. #3(#1 #4--#5 #2)#6}

\usepackage{bbm}

\usepackage{setspace}

\usepackage{tikz}
\usepackage{graphicx}
\usepackage{color}
\usepackage{soul}
 
\usepackage{glossaries}

\usepackage[normalem]{ulem}

\makeatletter
\g@addto@macro\bfseries{\boldmath}
\makeatother

\numberwithin{equation}{section}

\newtheorem{theorem}{Theorem}[section]
\newtheorem{definition}[theorem]{Definition}
\newtheorem{lemma}[theorem]{Lemma}
\newtheorem{proposition}[theorem]{Proposition}
\newtheorem{corollary}[theorem]{Corollary}

\newtheorem{remark}[theorem]{Remark}

\def\Sbar{{\overline{S}}}
\def\d{{\bm{d}}}

\newcommand{\Bin}[1]{\textup{Bin}\left( #1 \right)}

\def\Gnd{\mathcal{G}(\d)}

\def\Gnp{\mathcal{G}(n,p)}
\def\co{{J}}
\newcommand{\cop}[1][p]{\co(#1)}
\def\bo{{\bar{\co}}}
\def\Ds{\mathcal{D}_S}
\def\goodds{\mathcal{D}_S'}
\def\Dsp{\mathcal{D}}
\def\gooddsp{\mathcal{D}(p)}

\newcommand{\tikzoverset}[2]{%
	\tikz[baseline=(X.base),inner sep=0pt,outer sep=0pt]{%
		\node[inner sep=0pt,outer sep=0pt] (X) {$#2$}; 
		\node[yshift=1pt] at (X.north) {$#1$};
}}
\def\MJ{\tikzoverset{\sim}{M}}

\newcommand{\ytil}[1]{\tilde{y}_{#1}}
\def\Ssmall{S_{\rm{small}}}
\def\Sbig{S_{\rm{big}}}

\def\deg{\d}
\def\Cnd{\mathcal{C}(\d)}

\newcommand{\prob}{\mathbb{P}}
\newcommand{\Prob}[1]{\prob\left(#1\right)}

\newcommand{\CProb}[2]{\prob\left(\left.#1\right|#2\right)}

\newcommand{\expec}{\mathbb{E}}
\newcommand{\Exp}[1]{\expec\left[#1\right]}

\newcommand{\Var}[1]{\textup{Var}\left(#1\right)}

\newcommand{\Cov}[1]{{\mbox{Cov}\left(#1\right)}}

\newcommand{\aas}{a.a.s.}

\newcommand{\ind}[1]{\mathbbm{1}_{\left\{#1\right\}}}

\title{Large induced subgraphs of random graphs with given degree sequences}
\author{Angus Southwell
		\thanks{Research supported by Australian Government Research Training Program Scholarship.}\\
		{\small School of Mathematics}\\
 {\small Monash University}\\
 {\small VIC 3800 Australia}\\
  {\small {\tt  angus.southwell@monash.edu}}
		\and 
		Nicholas Wormald
		\thanks{Research supported by ARC Discovery Project DP180103684.}\\
		{\small School of Mathematics}\\
 {\small Monash University}\\
 {\small VIC 3800 Australia}\\
  {\small {\tt nick.wormald@monash.edu }}}
\date{}
\usepackage{bm}
\usepackage{bbold}

\begin{document}
	
\maketitle

\begin{abstract}
	We study a random graph $G$ with given degree sequence $\d$, with the aim of characterising the degree sequence of the subgraph induced on a given set $S$ of vertices. For suitable  $\d$ and $S$, we show that the degree sequence of the subgraph induced on $S$ is essentially concentrated around a sequence that we can deterministically describe in terms of $\d$ and $S$.  We then give an application of this result, determining a threshold for when this induced subgraph contains a giant component. We also apply a similar analysis to the case where $S$ is chosen by randomly sampling vertices with some probability $p$, i.e. site percolation, and determine a threshold for the existence of a giant component in this model. We consider the case where the density of the subgraph is either constant or slowly going to $0$ as $n$ goes to infinity, and the degree sequence $\d$ of the whole graph satisfies a certain maximum degree condition. Analogously, in the percolation model we consider the cases where either $p$ is a constant or where $p \to 0$ slowly. This is similar to work of Fountoulakis in 2007 and Janson in 2009, but we work directly in the random graph model to avoid the limitations of the configuration model that they used.
\end{abstract}

\section{Introduction}\label{sec:intro}

 	Random graphs with given degree sequences are a well-studied random graph model. To define the model, let $\d = (d(1), \dots, d(n))$ be the degree sequence of a graph. Then $\Gnd$ is a uniform random graph with degree sequence $\d$. This graph model has been the focus of much study recently, both due to improvements in the tools to study the model and also as it has found applications as a null model for studying networks (see \cite{KlamtEtAl09}). Compared to binomial random graphs, this model is much better suited to studying properties of graphs where the degrees of the vertices are not concentrated around a particular value. However, this comes at the cost of ease of analysis: events that are trivial to study in the binomial random graph model (such as the adjacency of two vertices) are quite non-trivial in $\Gnd$ and not fully understood in general. 
	
	In this paper we study induced subgraphs of random graphs with given degree sequences, i.e.\ the degree sequence of the subgraph $G[S]$ of $G\in \Gnd$ induced by $S\subseteq V(G)$. Our main results are that the degree sequence of the induced subgraph is close to a model degree sequence  $\d_H$ defined in \Cref{def:dh}. In particular, the distribution of the degree of a vertex in $G[S]$ is approximately binomial, in terms of its degree in $G$ and the density of $S$ in $G$ (see \Cref{sec:results} for a precise statement). We use this approximation to show that with probability tending to 1 as $n\to\infty$ (\aas),  the large entries	in the degree sequence of $G[S]$   are asymptotically equal to the  corresponding entries in $\d_H$, and the frequencies of small entries in each sequence are also close. We formally state this in \Cref{lem:dh-ds-comparison}. 
	
	The result mentioned above applies to a given subset $S=S(n)$. We also make use of it to prove a similar result about the model where $G$ is again a uniformly random graph with degree sequence $\d$, but where $S$ is chosen randomly by taking each vertex independently with probability $p=p(n)$. This model is commonly known as \emph{(site-)percolated random graph $\Gnd$ with survival probability $p$}. This is in contrast to \emph{bond percolation}, where edges are deleted instead of vertices. Percolation problems have been studied on a wide range of graphs, both deterministic and random, since the 1950s. See, for example, the work of \citet{BroadbentHammersley57,Fountoulakis07,Janson2009}, or \citet*{McDiarmidScottWithers2021}. In this paper, we use the phrase ``percolated random graph'' and the notation $G_\d(p)$ to refer to a uniformly random graph with degree sequence $\d$ after site percolation with survival probability $p$. In this model, $S$ is a random variable, where each subset $S\subset [n]$ occurs with probability $p^{|S|}(1-p)^{n-|S|}$. We define a model sequence $\d_A$ (formally given in \Cref{def:perc-seq}), which is a function of $\d$ and $p$, and show that, for suitable $\d$ and $p$, the degree sequence of the percolated random graph is close to the model sequence $\d_A$.  
	
	The relationship between the large entries of $\d_A$ and $\d_S$ in the percolated random graph model is less precise than the corresponding relationship between $\d_H$ and $\d_S$ in the model where $S$ is given. Instead of estimating each entry, we estimate the degree of a vertex conditional on it being a member of $S$. As well as this, we give a result about the sum of the large entries in each sequence.
	This is because of the potential lack of concentration of large degrees in the induced subgraph. For instance, if only one vertex $v$ exists of very large degree $i$, and the rest have small degree $j$, then the maximum degree in the induced subgraph will be close to $ip$ with probability  $p$ (in the case that $v\in S$), and at most $j$ otherwise. 

	A common problem studied in random graphs and percolation models is the existence of a giant component. \citet{Molloy-Reed1995} used the configuration model (often denoted $\Cnd$) to investigate the existence of a giant component in $\Gnd$. 
	The configuration model is a model proposed by \citet{Bollobas1980config} to study $\Gnd$ which constructs a random (pseudo)graph with the correct degree sequence from a random pairing of sets of points in bins corresponding to the vertices. It is much easier to analyse than $\Gnd$, but the need to transfer results to $\Gnd$ resulted in strict conditions on the degree sequence $\d$ in~\cite{Molloy-Reed1995}.
	Recently, \citet*{Joos2018} generalised this to fully describe the threshold for the existence of a giant component in $\Gnd$ in terms of $\d$, for all sequences $\d$.
	
	The results of \citet{Molloy-Reed1995} were used by \citet{Fountoulakis07} to study the threshold for the existence of a giant component in a percolated random graph. 
	Again this was done by studying the configuration model. A key element of his proof is the following fact: the percolated random graph is distributed uniformly at random conditioned on its degree sequence. He then studied the distribution of the resulting degree sequence in both site and bond percolation models. This result has strict requirements on $\d$, such as a maximum degree of at most $n^{1/9}$, bounded average degree, and a sufficiently nice limiting distribution.
	Fountoulakis then applied the aforementioned results of \cite{Molloy-Reed1995} to prove a threshold for the existence of a giant component in (site or bond) percolated $\Gnd$. Janson \cite{Janson2009} used similar ideas and tools from the theory of branching processes to prove a similar result for a wider range of degree sequences. 
	Recently, \citet*{Fountoulakis2016} used results in \cite{Joos2018} to prove results about the threshold for the existence of a giant component in bond percolated $\Gnd$. These results apply for a wider range of degree sequences than considered in \cite{Fountoulakis07} and \cite{Janson2009}, but also assume that $\Gnd$ has bounded average degree and that the survival probability $p \in (0,1]$ is a constant.

In this paper we apply the recent result of \citet{Joos2018} and results about our model degree sequences ($\d_H$ for when $S$ is fixed and $\d_A$ for when $S$ is random) to determine a threshold for the existence of a giant component in $G[S]$. This serves as an example of how our main result can be used to study induced subgraphs of $\Gnd$: by combining this result with known thresholds for properties of $\Gnd$, one can determine thresholds for these properties in $G[S]$.
Notably, our results allow for cases where $\Gnd$ has  maximum degree slightly less than $\sqrt{|E(G)|}$ (see \labelcref{eqn:conditions} for the precise condition),
as long as the density of the subgraph (relative to the whole graph $G$) is either bounded away from 0 and 1, or  goes to $0$ sufficiently slowly (roughly up to $n^{-\varepsilon}$ for some small constant $\varepsilon > 0$). We can achieve this extension of the results in \cite{Fountoulakis07} on site percolation by carrying out our analysis in $\Gnd$ directly, as opposed to using the configuration model, and we utilise the switching method heavily. In contrast to the results in \cite{Fountoulakis2016} on bond percolation, our results also apply in cases where $\Gnd$ has average degree only slightly less than  $\sqrt{|E(G)|}$   and the survival probability is at least $|E(G)|^{-\varepsilon'}$, for a small constant $\varepsilon'$.  In particular,  for nearly regular degree sequences $\d$  (e.g.\ typical degree sequences arising from $\Gnp$ or $\mathcal{G}(n,m)$), our results apply when the maximum degree is  $O(n^{1-\varepsilon})$ for any $\varepsilon>0$.  In upcoming work, we also use our degree sequence characterisation to prove thresholds for the connectivity of $G[S]$, as well as results on its chromatic number and its automorphism group.

\section{Main results}\label{sec:results}

	Here we give much of the notation we use, as well as describing our main results. Let $\d$ be a graphical sequence of length $n$, that is, let $\d = (d(1), \dots, d(n))$ be a sequence of non-negative integers such that there exists a graph with vertex set $[n] = \{1,\dots, n\}$ where each vertex $i \in [n]$ has degree $d(i)$. Without loss of generality, we assume that all entries of $\d$ are at least $1$ and are in non-decreasing order, so $1\leq d(1) \leq d(2) \leq \dots \leq d(n)$. We also define $\Delta  = \Delta(\d)$ to be the value of the largest entry in $\d$. For a set $A \subseteq [n]$, let $d(A) = \sum_{i \in A} d(i)$ be the \emph{total degree} of $A$. We also use $M(\d)$ to denote $d([n])$, and call it the \emph{total degree} of a sequence $\d$. For brevity, we use $M$ to denote $M(\d)$ where $\d$ is the degree sequence of the underlying random graph $G \in \Gnd$. We always use $S = \{i_1,\dots, i_s\} \subset [n]$ to denote the vertices of the induced subgraph of $\Gnd$, and we define $\Sbar = [n]\backslash S$ and $\gamma = \gamma(S) = d(S)/M$.

	\subsection{Subgraphs induced on a vertex set}
		Suppose $S=S(n)\subseteq [n]$, and suppose $(\d,S)$ satisfies
		\begin{align}\label{eqn:conditions}
			\Delta^2 (\gamma^{-1} \log M)^{12} \leq \delta d(S)
		\end{align}
		for some $\delta \to 0$ sufficiently slowly as $n\to\infty$ (equivalently, $M \to\infty$). Throughout the proofs, we use $\delta$ and various powers of it to bound the rate at which certain functions grow or shrink. We assume that $\delta = \Omega((\log\log M)^{-1})$, or equivalently that $\delta^{-1} = O(\log\log M)$. Define $\co = \delta^{-1/16} \gamma^{-1} \log M$. 
		We suppose that $\gamma < 1-c$ for some constant $c > 0$, but we allow $\gamma=\gamma(n) \to 0$. That is, $d(\Sbar) \geq cM$ for some constant $c > 0$, but $d(S) = o(M)$ is possible. The condition on $\gamma$  given in \labelcref{eqn:conditions} implies that 
		\begin{align*}
			\gamma \geq \delta^{-1/13}\frac{(\Delta^2 \log^{12}M)^{1/13}}{M^{1/13}}.
		\end{align*}
		This immediately implies that $\gamma = \omega(M^{-1/13})$, a fact used throughout the proofs. 
		The powers of $\log M$ in our definitions and results are not necessarily optimised, either for studying the distribution of the induced degree sequence in general or for studying the threshold for the existence of giant components.
		
		We also note that the conditions given in \labelcref{eqn:conditions} imply that $\Gnd$ is non-empty. 
		
		\begin{proposition}\label{prop:graphical}
			If $\d$ is a sequence of length $n$ with even sum such that there exists a set $S \subset [n]$ satisfying \eqref{eqn:conditions}, then $\d$ is graphical, that is, there exists a graph with degree sequence $\d$.
		\end{proposition}
		\begin{proof}
			The inequality \eqref{eqn:conditions} implies that $\Delta^2 = o(d(S))$, which implies that $\Delta^2 = o(M)$ (since $\gamma \leq 1$). \citet{Koren73} (Section 1) states that if a sequence $\d$ is not graphical, then there exist disjoint, non-empty sets $A, B \subset [n]$ such that 
			\begin{align*}
			\sum_{i \in A} d(i) - \sum_{j \in B} d(j) > a(n-1-b),
			\end{align*}
			where $a = |A|$ and $b = |B|$. Suppose that such sets $A$ and $B$ existed. The left hand side of this inequality is at most $a\Delta$. Thus, this inequality could only be true if $b > n-1-\Delta$. This implies that $a < \Delta + 1$, and also that $\sum_{j \in B}d(j) = \sum_{j \in [n]} d(j) - \sum_{j \notin B} d(j) \geq M(\d) - \Delta^2$. Since $\sum_{i \in A}d(i) \leq \Delta^2 = o(M)$, it follows that the left hand side tends to $-\infty$ as $n \to \infty$, which is a contradiction. Therefore the inequality cannot hold, and the sequence is graphical.
		\end{proof}

		In view of this lemma, by supposing that $\d$ is a sequence of length $n$ with all entries at least 1 and even sum that satisfies \labelcref{eqn:conditions}, we may assume that $\d$ is graphical, which is useful when talking about probabilities in associated random graph models.
		We next define a deterministic sequence $\d_H$ that in some sense represents a typical degree sequence of $G[S]$. Let $\d_S$ be the degree sequence of the graph $G[S]$. For an arbitrary sequence $\d$, let $n_k(\d)$ be the number of entries of $\d$ that are equal to $k$.
		
		\begin{definition}\label{def:dh} 
			Let $d$, a sequence of length $n$,
			and a set $S = \{i_1,\dots, i_s\} \subset [n]$ be given. To define $\d_H=\d_H(S)$, let
				  $Z_j \sim \Bin{j, \frac{d(S)}{M}}$ and define
				\begin{align*}
					N(k) = \left\lfloor \sum_{i \in S} \Prob{Z_{d(i)} \leq k} + \frac{1}{2} \right\rfloor
				\end{align*}
				for $k\geq 0$, and $N(-1) = 0$. Then
				define $\d_H$  to be the non-decreasing sequence in which $n_k(\d_H)$, i.e. the number of occurrences of $k$ in $\d_H$, is given by $n_k(\d_H) = N(k) - N(k-1)$.  
		\end{definition}
		
		We note that $\d_H$ may not be a graphical sequence, but we do not need it to be. The main result on the degree sequence of $G[S]$ is the following theorem, in which each of the sequences $\d_S$ and $\d_H$ is essentially segmented into two parts (with some overlap to ensure that all entries of both sequences are covered by the theorem). Part (a) of the theorem implies that, beyond a certain index, the corresponding entries in the two sequences are \aas\ asymptotic to each other. It also gives an explicit formula for these entries which would be suggested by a naive intuition based on expectation, and is useful for practical purposes. Below this index, it is difficult to obtain asymptotic values for each entry of $\d_S$, so we just give rough bounds in (b), and a distributional result in (c). The latter, applying for a slightly larger range than (b) in order to overlap with the range for which part (a) applies, states that the number entries that are equal to any given $k\le \frac{1}{2}\gamma\co$ is similar in each sequence. 
			
		\begin{theorem}\label{lem:dh-ds-comparison}
			Let $\d$ be a sequence of length $n$ with all entries at least 1 and even sum, and let $S \subset [n]$ be such that $(\d, S)$ satisfies \labelcref{eqn:conditions} for some $\delta \to 0$ and $\gamma < 1-c$ for some constant $c > 0$. The following claims hold with probability $1 - O(1/\sqrt{\log M})$:
			\begin{enumerate}[label=(\alph*)]
				\item    $d_S(k) = \gamma d(i_k)(1 \pm 8\delta^{1/64}) = d_H(k) \left( 1\pm 12\delta^{1/64} \right)$ for all $k$ such that $d(i_k) \geq \delta^{1/32}\co$;
				\item   $\max\{d_S(k), d_H(k)\} \leq 2\gamma\delta^{1/32}\co$  for all $k$ such that $d(i_k) < \delta^{1/32}\co$;
				\item   $|n_i(\d_S) - n_i(\d_H)| \leq \frac{\gamma n_{i}(\d_H)}{\co^{3}} + \gamma\co^5$ for all $i  \leq \frac{1}{2}\gamma \co$.
			\end{enumerate}
		\end{theorem}

		 We apply this and the results of \citet{Joos2018} to prove the following threshold for the existence of giant components in induced subgraphs of $\Gnd$. 

		\begin{theorem}\label{thm:main-gc-pretty}
			Let $\d$ be a sequence of length $n$ with all entries at least $1$ and even sum, and let $S$ be a subset of $[n]$. Let $\gamma = d(S)/M$, and suppose that $\Delta^2 \gamma^{-12}\log^{12} M = o(\gamma M)$.
			Then $G[S]$ \aas\ contains $(|S| - n_0(\d_H))(1 + o(1))$ non-isolated vertices. Furthermore, $G[S]$ \aas\ contains a component on a positive fraction of the non-isolated vertices if and only if there exists some constant $\varepsilon > 0$ such that $R(\d_H) \geq \varepsilon \gamma^2 M$.
		\end{theorem}
	
		We prove this result by applying \Cref{lem:dh-ds-comparison} and the result of \citet{Joos2018} about the threshold for giant components in $\Gnd$ (formally stated in \Cref{thm:joos}). We show that for two sequences that are close in the sense described in \Cref{lem:dh-ds-comparison}, the thresholds for the existence of a giant component coincide. We defer the proof of \Cref{lem:dh-ds-comparison} for now, and give this and all the intermediate results in \Cref{sec:dist}.

	\subsection{Random induced subgraphs of $G$}
	
		Now we consider the (site-)percolated random graph model $G_\d(p)$, for some $p \in (0,1)$. In this model, $S$ is a random variable where, for each subset $T \in [n]$, $\Prob{S = T} = p^{|T|} (1 - p)^{n - |T|}$. 
		Thus, the subgraph $G[S]$ is the subgraph of a uniformly random $G\in \Gnd$ induced on $S$, where $S$ is chosen by randomly keeping vertices of $G$ independently with some probability $p$, and deleting the rest.
		Again we suppose that $\d$ is ordered in non-decreasing order with all entries at least $1$.  Analogously to the case where $S$ is fixed, we impose the condition that 
		\begin{align}\label{eqn:conditions-perc}
			\Delta^2 (p^{-1} \log M)^{12} \leq \delta p M
		\end{align}
		for some $\delta \to 0$ sufficiently slowly as $n\to\infty$. We now define a model degree sequence of the percolated random graph $G_\d(p)$.
	
		\begin{definition}\label{def:perc-seq}
			Let $\d = (d(1),\dots, d(n))$ be a non-decreasing sequence of length $n$. Let $p \in (0,1)$, and let $X_j \sim \Bin{j,p}$. For $k \in \left\{0, \dots, \co\right\}$, define 
			\begin{align*}
				\tilde{N}(k) &:= \left\lfloor p\sum_{i \in V} \Prob{X_{d(i)} \leq k} + \frac{1}{2} \right\rfloor
			\end{align*}
			for $k \geq 0$, and $\tilde{N}(-1) = 0$. Then define $n_k(\d_A) = \tilde N(k) - \tilde N(k-1)$ to be the number of entries in $\d_A$ with value $k$. 
		\end{definition}
		
		Now we state the main result of our paper for degree sequences of site-percolated $\Gnd$.
		
		\begin{theorem}\label{lem:degree-sequence-concentration-perc}
		Let $\d$ be a sequence of length $n$ with even sum and all entries at least 1, and let $p \in (0,1)$ be such that $p < 1 - \varepsilon$ for some constant $\varepsilon > 0$ and $\Delta^2 p^{-12}\log^{12} M \leq \delta p M$ for some $\delta \to 0$. Then the following statements hold with probability $1 - o(1)$ in the percolated random graph $G_\d(p)$.
			\begin{enumerate}[label=(\alph*)]
				\item $|S| = np\left( 1 \pm 3 \sqrt{\frac{\log n}{pn}} \right)$.
				\item $d(S) = pM\left(1 \pm \frac{p^2}{M^{1/4}} \right)$. 
				\item $d_S(v) = pd(v) \left( 1 \pm 9 \delta^{1/64}  \right)$ for all $v \in S$ such that $d(v) > 2\delta^{1/32} \cop$.
				\item For all $i \leq \frac{1}{3}p\cop$,
				\begin{align*}
					|n_i(\d_S) - n_i(\d_A)| \leq  \frac{pn_i(\d_A)}{\cop^{3}}(1 + o(1)) + \frac{p\cop^{6}}{\sqrt{\log M}}.
				\end{align*}
			\end{enumerate}
		\end{theorem}
		Analogously to the definition of $\ytil{i}$, we also define
		\begin{align*}
			\tilde{w}_k := p\sum_{i \in V} \Prob{X_{d(i)} = k}.
		\end{align*}
		It follows immediately that $n_k(\d_A) = \tilde{w}_k \pm 1$. We also analogously define $\cop := \delta^{-1/16} p^{-1} \log M$. For $p=\gamma$ this is equivalent to the definition of $\co$ used previously. This also allows us to consider $\cop[\gamma(S)]$, the corresponding value of $\co$ for a given set $S \subset [n]$. This is useful as we often prove results for the percolation model by conditioning on a ``nice'' choice of $S$ and then applying results proved in the case where $S$ is fixed. As such, when proving results about the percolation model, we often write definitions from the previous section (e.g. $\gamma$, $Z_j$, $\d_H$, $\ytil{i}$) with the extra argument of $S$ to highlight the conditional probability space on which we define them.
		As in the case where $S$ is fixed, we apply \Cref{lem:degree-sequence-concentration-perc} to determine the threshold for the existence of a giant component in the site-percolated random graph under the conditions given in \labelcref{eqn:conditions-perc}. \
		
		\begin{theorem}\label{thm:perc-gc-nice}
			Let $\d$ be a sequence of length $n$ with even sum and all entries at least 1. Let $p \in (0,1)$ be such that $\Delta^2 p^{-12} \log^{12} M = o(p M)$. Let $G_\d(p)$ be the site-percolated random graph, where $G \sim \Gnd$.
			Then $G_\d(p)$ \aas\ contains $(np - n_0(\d_A))(1 + o(1))$ non-isolated vertices. Furthermore, $G_\d(p)$ \aas\ contains a component on a positive fraction of the non-isolated vertices if and only if $R(\d_A) \geq \varepsilon p^2M$ for some constant $\varepsilon > 0$. 
		\end{theorem}	
	
		Much like in the case where $S$ is fixed, we prove this by showing that $\d_S$ and $\d_A$ are \aas\ close, and then showing that for sequences that are close these thresholds coincide. We give the proof of \Cref{lem:degree-sequence-concentration-perc} in \Cref{sec:perc}.	
	
\section{Distribution of the induced vertex degree}\label{sec:dist}

	In this section we prove \Cref{lem:dh-ds-comparison}. Let $A_v^{i}$ denote the set of $G \in \Gnd$ such that $d_S(v) = i$.
	
	\begin{lemma} \label{lem:degree-switching}
		Let $v$ be an arbitrary vertex in $S$. Then 
		\begin{align*}
			\frac{|A_v^{i}|}{|A_v^{i+1}|} = \frac{i+1}{d(v)-i} \cdot \frac{d(\Sbar)}{d(S)} \left( 1 + O\left( \frac{\Delta^2}{d(S)} \right) \right).
		\end{align*}
	\end{lemma}
	\begin{proof}
		We define an operation called a switching that takes a graph $G \in A_v^{i+1}$ to some $G' \in A_v^i$. Let $G \in A_v^{i+1}$. To perform a switching, choose a vertex $y$ such that $vy \in E(G)$ and $y \in S$, as well as an ordered pair of vertices $(u,x)$ such that $ux \in E(G)$ and $u \in \Sbar$ ($x$ can be in either $S$ or $\Sbar$). It is also required that
\begin{enumerate}[label=(\alph*)]
			\item the vertices $\left\{u,v,x,y\right\}$ are distinct, and
			\item $xy \notin E(G)$ and $uv \notin E(G)$. 
\end{enumerate}
The switching deletes edges $vy$ and $ux$, replacing these edges with $uv$ and $xy$ and hence creating a new multigraph $G^\prime$, and the conditions  (a) and (b) imply    that $G' \in A_v^{i}$. 
		This switching is illustrated in \Cref{fig:switching-degrees}.
		\begin{figure}[h!]
			\begin{center}
				\begin{tikzpicture}[line width=.5pt,vertex/.style={circle,inner sep=0pt,minimum size=0.2cm}, scale=.6]
					
					% % % vertices

					\node [black, fill=black,label={[label distance=0mm]45:$y$}] (y) at (45:2)  [vertex]{}; 
					\node [draw, black, fill=white,label={[label distance=0mm]135:$v$}] (v) at (135:2)  [vertex]{}; 
					\node [black, fill=black,label={[label distance=0mm]225:$u$}] (u) at (225:2)  [vertex]{};
					\node [black, fill=black,label={[label distance=0mm]315:$x$}] (x) at (315:2)  [vertex]{}; 
					\node [label={[label distance=0mm]90:$G \in A^{i+1}_{v}$}] at (90:2.5) [vertex]{};

					% % % lines
					
					\draw[thick, -] (y) to (v);
					\draw[thick, -] (x) to (u);
					\draw[thick, dashed] (u) to (v);
					\draw[thick, dashed] (y) to (x);
					
					\draw[thick, ->] (2.5,0) to (3.5,0);

				\end{tikzpicture}
				\begin{tikzpicture}[line width=.5pt,vertex/.style={circle,inner sep=0pt,minimum size=0.2cm}, scale=.6]

					% % % vertices

					\node [black, fill=black,label={[label distance=0mm]45:$y$}] (y) at (45:2)  [vertex]{}; 
					\node [draw, black, fill=white,label={[label distance=0mm]135:$v$}] (v) at (135:2)  [vertex]{};
					\node [black, fill=black,label={[label distance=0mm]225:$u$}] (u) at (225:2)  [vertex]{};
					\node [black, fill=black,label={[label distance=0mm]315:$x$}] (x) at (315:2)  [vertex]{}; 
					\node [label={[label distance=0mm]90:$G' \in A^{i}_{v}$}] at (90:2.5) [vertex]{};

					% % % lines
					
					\draw[thick, dashed] (y) to (v);
					\draw[thick, dashed] (x) to (u);
					\draw[thick, -] (u) to (v);
					\draw[thick, -] (y) to (x);

				\end{tikzpicture}
				
			\end{center}
			\caption{A switching. Here $v,y \in S$ and $u \in \Sbar$. Edges present in $G$ (on the left, respectively $G'$ on the right) are given as solid lines, forbidden edges are given as dashed. Other edges may be present or absent.}
			\label{fig:switching-degrees}
		\end{figure}
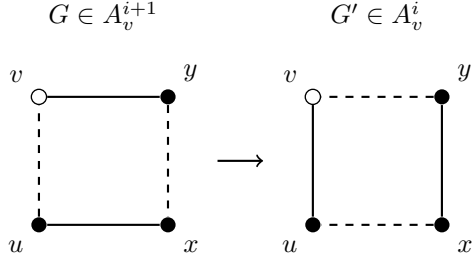	
		
		Now we find upper and lower bounds on the number of switchings that create a particular $G' \in A_v^{i}$. 
		Given  $G \in A_v^{i+1}$, there are $i+1$ choices for a vertex $y$ such that $vy \in E(G)$ and $y \in S$. There are $d(\Sbar)$ choices for a vertex $u \in \Sbar$ and neighbour $x$. Thus, there are at most $(i+1)d(\Sbar)$ switchings that take $G \in A_{v}^{i+1}$ to some $G^\prime \in A_{v}^i$. To determine a corresponding lower bound, we note that since $G$ has maximum degree at most $\Delta$, the number of choices for $\{u,x,y\}$ as described above that violate (a) is $O((i+1)\Delta)$ and for (b) it is $O((i+1)\Delta^2)$. Hence the number of valid switchings that can be applied to each $G \in A_v^{i+1}$ is $(i+1)(d(\Sbar) +O(\Delta^2))$.
		
		Now we use a very similar argument to count the switchings that create a particular $G' \in A_v^{i}$.  There are $d(v)-i$ choices for the vertex $u$, and $d(S)$ choices for an ordered pair of vertices $(x,y)$ such that $y \in S$ and $xy \in E(G')$. So an upper bound is $(d(v)-i)d(S)$. The number of these combinations that are invalid, due to a vertex being repeated or the edge $ux$ or $vy$ being present,  is  $O(\Delta^2)(d(v)-i)$. It follows that the number of switchings that create $G'$ is $(d(v)-i)(d(S)+O(\Delta^2))$.
			
		From the conclusions of the previous two paragraphs, the total number of switchings applicable to graphs in $A_v^{i+1}$ can be counted in two different ways as $|A_v^{i+1}| (i+1)(d(\Sbar) +O(\Delta^2))$ and $|A_v^{i}|(d(v)-i)(d(S)+O(\Delta^2))$. The lemma follows, since $d(\Sbar) = \Theta(M)$ and $\Delta^2 = o(d(S))$.
	\end{proof}

	Define $\Ssmall = \{i_1,\dots, i_\ell\}$ where $\ell$ is the smallest index such that $d(i_j) > \co$ for all $j > \ell$. That is, $\Ssmall$ is the set of vertices in $S$ with degree (in $G$) at most $\co$. Naturally, we can also define $\Sbig = S \backslash \Ssmall$. Define 
	\begin{align*}
	\ytil{i} = \sum_{v \in \Ssmall} \Prob{Z_{d(v)} = i}.
	\end{align*}
	For small $i$ (that is, smaller than $c\gamma \co$ for some constant $c < 1$), $\ytil{1}$ is very close to $n_i(\d_H)$, the number of entries in the sequence $\d_H$ with value $i$. One noteworthy and straightforward consequence of \Cref{lem:dh-ds-comparison} is that $n_i(\d_H) = \ytil{i} \pm (1 + o(M^{-5}))$ for all $i \leq \frac{1}{2}\gamma \co$. For the sequence $\d_S$, we can analogously define
	\begin{align}\label{def:Yi}
	Y_i = \sum_{v \in \Ssmall} \ind{d_S(v) = i}. 
	\end{align}	
	These definitions allow us to consider the behaviour of vertices in $\Ssmall$ and $\Sbig$ somewhat separately. This is useful in certain circumstances, particularly when studying the distribution of the number of vertices with very low degrees (e.g. 0, 1, or 2) in $\d_H$ and $\d_S$, as the number of vertices in $\Sbig$ with very low induced degree is \aas\ $0$. More specifically, \Cref{lem:dh-ds-comparison} implies that \aas\ $n_i(\d_S) = Y_i$ for all $i \leq \frac{1}{2}\gamma\co$. A later result (\Cref{lem:binomial-degrees}) then implies that $\Exp{Y_i} = \ytil{i} \left( 1 + o(\co^{-6})\right)$ for all $i \leq \co$. 
	\begin{remark}\label{rem:yi-ratio}
		Let $S_j$ be the set of $i \in S$ such that $d(i) = j$. 
		Then for $k \in \left[1, \frac{1}{2}\gamma \co\right]$, 
		\begin{align*}
		\ytil{k} &= \sum_{i \in \Ssmall} \Prob{Z_{d(i)} = k} =
		\sum_{j\leq \co} |S_j| \Prob{Z_j = k}  = \sum_{j \leq \co} \frac{d(S)}{d(\Sbar)} \frac{j-k+1}{k} |S_j| \Prob{Z_j = k-1} .
		\end{align*}
		With some naive bounds on the value of $\frac{j-k+1}{k}$, this gives useful bounds on the ratios between successive values of $\ytil{i}$, and thus on $n_i(\d_H)$. Since $j \leq \co$, 
		\begin{align*}
		\ytil{k} &= \sum_{j\leq \co} |S_j| \Prob{Z_j = k} \leq \frac{\co}{k} \frac{d(S)}{d(\Sbar)} \sum_{j \leq \co} |S_j| \Prob{Z_j = k-1} = \frac{\gamma \co}{k(1-\gamma)}\ytil{k-1}.
		\end{align*}
		Thus, $\ytil{k} = O(\gamma\co\ytil{k-1})$. More commonly, we use the form $\ytil{k-1} = \Omega\left( \frac{\ytil{k}}{\gamma\co} \right)$.   
	\end{remark}

\subsection{Concentration of large degrees}

	Recall that $d_S(v)$ is the degree of vertex $v$ in $G[S]$.  Also recall the definition of $\gamma = d(S)/M$, and that $\delta$ is an arbitrary function such that $\delta \to 0$ and $\Delta^2 (\gamma^{-1} \log M)^{12} \leq \delta d(S)$.

	\begin{lemma} \label{lem:epsilon}
		Suppose $\varepsilon = 4\delta^{1/64}$ and define $i_0= i_0(v) = \gamma d(v)$ (not necessarily an integer). Then, for $n$ sufficiently large, 
		\begin{align*}
			\prob\big(d_S(v) \in [i_0(1 - 2\varepsilon), i_0(1 + 2\varepsilon)]\big) < 2d(v) \exp\left( -\varepsilon^2 i_0/2 \right).
		\end{align*}
	\end{lemma}
	
	\begin{proof} 
		We first prove that the probability that $d_S(v) < i_0(1 - 2\varepsilon)$ is less than $d(v)\exp\left( -\frac{1}{2} \varepsilon^2 i_0 \right)$. 
		Define $i_k = (1 - k\varepsilon)i_0$ for all $k > 0$. Recall that $|A_v^{i}|$ is the number of graphs in $\Gnd$ such that $d_S(v) = i$. For all $i \leq i_1 - 1$, \Cref{lem:degree-switching} implies that
		\begin{align*}
			\frac{|A_v^{i}|}{|A_v^{i+1}|} &= \frac{i+1}{d(v)-i}\frac{d(\Sbar)}{d(S)}\left( 1 + O\left( \frac{\Delta^2}{d(S)} \right) \right)
			\leq \left( 1 - \varepsilon \right)\frac{i_0}{d(v) - i_0} \frac{d(\Sbar)}{d(S)}\left( 1 + O\left( \frac{\Delta^2}{d(S)} \right) \right).
		\end{align*}
		By definition of $i_0$, 
		\begin{align*}
			\frac{i_0}{d(v) - i_0} = \frac{\frac{d(S)}{M}}{1 - \frac{d(S)}{M}} = \frac{d(S)}{d(\Sbar)}.
		\end{align*}
		Thus, for all $i \leq i_1 - 1$, 
		\begin{align}\label{eqn:epsilon}
			\frac{|A_v^{i}|}{|A_v^{i+1}|} \leq (1-\varepsilon) \left( 1 + O\left( \frac{\Delta^2}{d(S)} \right) \right) < 1 - \frac{3}{4}\varepsilon,
		\end{align}
		for $n$ sufficiently large, since $\Delta^2 = o(\varepsilon d(S))$. Hence for all $i \leq \lceil i_2 \rceil - 1$,  
		\begin{align*}
			\frac{|A_v^{i}|}{|A_v^{\lfloor i_1\rfloor}|} &\leq \left(1 - \frac{3}{4}\varepsilon\right)^{\varepsilon i_0-1}
			< \exp\left( - \frac{2}{3}\varepsilon^2 i_0 + O\left( \varepsilon^3 i_0 \right) \right)
			< \exp\left( -\frac{1}{2}\varepsilon^2 i_0 \right),
		\end{align*}
		where the last inequality holds for $n$ sufficiently large. So if $i \leq \lceil i_2 \rceil -1$, it follows that $\Prob{ d_S(v) = i} <  \exp\left( -\frac{1}{2}\varepsilon^2 i_0 \right)$. Performing a union bound over all possible induced degrees $ i \leq i_2$ gives that 
		\begin{align*}
			\Prob{ d_S(v) \leq i_2} <  d(v)\exp\left( -\frac{1}{2}\varepsilon^2 i_0 \right).
		\end{align*}

		The argument for the upper bound is obtained symmetrically mutatis mutandis, and the lemma follows from the union bound.
	\end{proof}

	\begin{lemma}\label{cor:high-degree-percolation}\label{cor:high-degree-crossover}
		As in \Cref{lem:epsilon}, suppose $\varepsilon = 4\delta^{1/64}$. The probability that
		\begin{align*}
			d_S(v) \in \left[ \gamma d(v)(1 - 2\varepsilon), \gamma d(v)(1 + 2\varepsilon) \right]
		\end{align*}
		for all vertices $v \in S$ such that $d(v) > \delta^{-1/32}\gamma^{-1} \log M$
		is $1 - o\left( M^{-5} \right)$.
	\end{lemma}
	\begin{proof}
		We apply \Cref{lem:epsilon} along with the union bound over all vertices $v \in S$ such that $d(v) > \delta^{-1/32}\gamma^{-1} \log M$. \Cref{lem:epsilon} implies that the probability that $d_S(v)$ is outside the specified range is at most $2n\exp\left( -\frac{1}{2}\varepsilon^2 i_0 \right)$. Note that by assumption, $i_0 > \delta^{-1/32}\log M$, where $\delta \to 0$. Combining this with the union bound implies that the probability that there exists some vertex $v$ with degree greater than $\delta^{-1/32}\gamma^{-1} \log M$ such that $d_S(v)$ is outside its specified range is at most $2n^2\exp\left( -\frac{1}{2}\varepsilon^2 \delta^{-1/32}\log M \right) = 2n^2 M^{-8}$. Since $n \leq M$, the claim holds. 
	\end{proof}

\subsection{Distribution   of small vertex degrees}

	Define $\bo = \min\{\co, \Delta\}$. Using this notation simplifies some arguments by allowing us to combine cases where $\co \leq \Delta$ and $\co > \Delta$. 	
	In the following lemma, we use a slightly more complicated switching than in \Cref{lem:degree-switching}, moving three edges instead of the usual two. The reason for this is that we wish to preserve the degrees of two adjacent vertices, $v_1$ and $v_2$, in $G[S]$, while switching away the edge between them. To use the previous switching, this means that the two other adjacent vertices in the switching must be in $S$, and adjacent. Possible variations in the number of choices of such a pair of adjacent vertices would cause a problem. Instead, we use a trick with its origins in the switchings introduced by McKay and Wormald~\cite{MWlow}, whereby a third pair of vertices are involved in order to make the number of switchings much more stable.  

	\begin{lemma}\label{lem:sd-adjacency-probability}
		Suppose $(\d, S)$ satisfies condition \labelcref{eqn:conditions}. Let $k$ be fixed, let $\{v_1, \dots, v_k\} \subset \Ssmall$, and let $G$ be a uniformly random graph with degree sequence $\d$. Then 
		\begin{align*}
			\CProb{v_1v_2 \in E(G)}{d_S(v_1) = i_1, \dots, d_S(v_k) = i_k} = O\left( \frac{\bo^2M}{d(S)^2} \right)
		\end{align*}
		for all $i_j \leq d(v_j)$ for $j \leq k$, and $\Prob{v_1v_2 \in E(G)} = O\left( \frac{\bo^2M}{d(S)^2} \right)$. 
	\end{lemma}
	\begin{proof}
		First note that if one of $i_1$ or $i_2$ is equal to $0$, then the probability in question is 0. Thus we may suppose that $i_1,~i_2 > 0$. 
		
		Let $A_{v_1, v_2}$ be the subset of $\Gnd$ consisting of the graphs where $v_1$ and $v_2$ are adjacent and each vertex $v_j$ has induced degree $i_j$ for $j \leq k$. Similarly, let $B_{v_1, v_2}$ be the subset of $\Gnd$ where $v_1$ and $v_2$ are not adjacent and each vertex $v_j$ has induced degree $i_j$ for $j \leq k$. 
		We define a switching between $A_{v_1, v_2}$ and $B_{v_1, v_2}$ as follows. Suppose $G \in A_{v_1, v_2}$. To perform a switching, choose two ordered pairs of vertices in $V(G)$, $(x,y)$ and $(a,b)$, such that $a b,\ x y \in E(G)$, and $y,\ b \in S$, and with the additional requirements that
	\begin{enumerate}[label=(\alph*)]
			\item the vertices $\left\{ v_1,v_2, a,b,x,y \right\}$ are distinct, with the exception that $y = b$ is permissible,  
			\item $v_1y,\ ax,\ v_2b \notin E(G)$, and
			\item  the degrees of $v_1,\dots, v_k$ in $G[S]$ are unchanged by the switching.
		\end{enumerate}

The switching deletes the edges $v_1v_2,\ xy,\ ab$ and replaces them with $v_1y,\ ax,\ v_2b$, creating a graph  $G' \in B_{v_1, v_2}$. 
A diagram illustrating this switching is given in \Cref{fig:switching-sd-adjacency}.  

		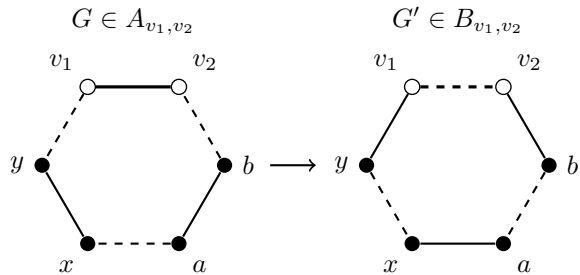
\begin{figure}[h!]
			\begin{center}
				\begin{tikzpicture}[line width=.5pt,vertex/.style={circle,inner sep=0pt,minimum size=0.2cm}, scale=.6]
					
					% % % vertices

					\node [black, fill=black,label={[label distance=0mm]0:$b$}] (y) at (0:2)  [vertex]{}; 
					\node [draw, black, fill=white,label={[label distance=0mm]60:$v_2$}] (v) at (60:2)  [vertex]{}; 
					\node [draw, black, fill=white,label={[label distance=0mm]120:$v_1$}] (u) at (120:2)  [vertex]{};
					\node [black, fill=black,label={[label distance=0mm]180:$y$}] (x) at (180:2)  [vertex]{}; 
					\node [black, fill=black,label={[label distance=0mm]240:$x$}] (a) at (240:2) [vertex]{};  
					\node [black, fill=black,label={[label distance=0mm]300:$a$}] (b) at (300:2)  [vertex]{}; 
					\node [label={[label distance=0mm]90:$G \in A_{v_1,v_2}$}] at (90:2.5) [vertex]{};
					
					% % % lines
					
					\draw[very thick, black,-] (u) to (v);
					\draw[thick, -] (x) to (a);
					\draw[thick, -] (y) to (b);
					\draw[thick,dashed] (u) to (x);
					\draw[thick, dashed] (b) to (a);
					\draw[thick, dashed] (y) to (v);
					
					\draw[thick, ->] (3,0) to (4,0);

				\end{tikzpicture}
				\begin{tikzpicture}[line width=.5pt,vertex/.style={circle,inner sep=0pt,minimum size=0.2cm}, scale=.6]
					
					% % % vertices
					
					\node [black, fill=black,label={[label distance=0mm]0:$b$}] (y) at (0:2)  [vertex]{}; 
					\node [draw, black, fill=white,label={[label distance=0mm]60:$v_2$}] (v) at (60:2)  [vertex]{}; 
					\node [draw, black, fill=white,label={[label distance=0mm]120:$v_1$}] (u) at (120:2)  [vertex]{};
					\node [black, fill=black,label={[label distance=0mm]180:$y$}] (x) at (180:2)  [vertex]{}; 
					\node [black, fill=black,label={[label distance=0mm]240:$x$}] (a) at (240:2) [vertex]{};  
					\node [black, fill=black,label={[label distance=0mm]300:$a$}] (b) at (300:2)  [vertex]{}; 
					\node [label={[label distance=0mm]90:$G' \in B_{v_1,v_2}$}] at (90:2.5) [vertex]{};

					% % % lines
					
					\draw[very thick, black, dashed] (u) to (v);
					\draw[thick,-] (u) to (x);
					\draw[thick, -] (b) to (a);
					\draw[thick, -] (y) to (v);
					\draw[thick, dashed] (x) to (a);
					\draw[thick, dashed] (y) to (b);

				\end{tikzpicture}
				
			\end{center} 
			\caption{A switching. Present edges are given as solid lines, forbidden edges are given as dashed. Other edges can be present or absent.}
			\label{fig:switching-sd-adjacency}
		\end{figure}
		First we determine a lower bound on the number of switchings that can be applied to each $G \in A_{v_1, v_2}$. 
		Since $y$ and $b$ are in $S$, there are $d(S)^2$ choices for $\{a,b,x,y\}$ ignoring the constraints (a) -- (c).
		Noting that the induced degrees of $v_1, \dots, v_k$ are unchanged by the switching if $b,y\in S$ and $\{v_1, \dots, v_k, a,b,x,y\}$ are distinct, we see that the number of choices for $\{v_1, v_2, a,b,x,y\}$ that violate (a) or (c) is $O(d(S) k \bo)$, and the number of choices that violate (b) is $O(d(S)\Delta^2)$. Since $k$ is fixed, this implies that the number of switchings that can be applied to each $G \in A_{v_1, v_2}$ is $d(S)(d(S) - O(\Delta^2))$. 
		
		Now we determine an upper bound on the number of switchings that create a particular $G' \in B_{v_1, v_2}$. 
		The definition of $B_{v_1, v_2}$ implies that the number of choices for $y$ is $i_1$, and similarly the number of choices for $b$ is $i_2$. The number of choices for the adjacent pair $(a,x)$ is at most $M$. Thus, the number of switchings that create $G'$ is at most $i_1i_2M$. 

		Combining these two bounds gives
		\begin{align*}
			\frac{|A_{v_1, v_2}|}{|B_{v_1, v_2}|} &\leq \frac{i_1 i_2M}{d(S)^2 } \left( 1 + O\left( \frac{\Delta^2}{d(S)} \right) \right).
		\end{align*}
		Since $\Delta^2 = o(d(S))$ by assumption, the multiplicative error term is $1 + o(1)$. Thus, the probability that the vertices $v_1$ and $v_2$ are adjacent, conditional on the induced degrees of $\{v_1, \dots, v_k\}$, is at most
		\begin{align*}
			\CProb{v_1 v_2 \in E(G)}{d_S(v_1) = i_1, \dots, d_S(v_k) = i_k} &= \frac{|A_{v_1, v_2}|}{|A_{v_1, v_2}| + |B_{v_1, v_2}|}
			\leq \frac{\bo^2 M}{d(S)^2} (1 + o(1)),
		\end{align*}
		since both $i_1$ and $i_2$ are at most $d(v_1)$ and $d(v_2)$ respectively and $v_1, v_2 \in \Ssmall$. This proves the first claim, and the second claim follows immediately from the law of total probability.
	\end{proof}

	\begin{lemma}\label{lem:binomial-degrees} \label{lem:low-degree-adjacency}
		Suppose $(\d,S)$ satisfies condition \labelcref{eqn:conditions}. Let $k$ be fixed and $\{v_1, \dots, v_k\} \subset \Ssmall$.
		Let $Z_j \sim \Bin{j, \frac{d(S)}{M}}$. Then, uniformly for all $i_1,\ldots, i_k\le \bo$, we have
		\begin{align*}
			\Prob{d_S(v_1) = i_1, \dots, d_S(v_k) = i_k} &= \prod_{j=1}^k \Prob{d_S(v_j) = i_j} \left( 1 + O\left( \bo \left(\frac{\Delta^2}{d(S)} + \frac{\bo^2 M}{d(S)^2} \right) \right) \right) \\
			&= \prod_{j=1}^k \Prob{Z_{d(v_j)} = i_j} \left( 1 + O\left( \bo \left(\frac{\Delta^2}{d(S)} + \frac{\bo^2 M}{d(S)^2} \right) \right) \right).
		\end{align*}
	\end{lemma}
	
	\begin{proof}
		We condition on the event that $d_S(v_j) = i_j$ for all $j \geq 2$, for an arbitrary choice of $(i_2, \dots, i_k)$ where $i_j \leq d(v_j)$. It suffices to show that 
		\begin{align*}
			\CProb{d_S(v_1) = i_1}{d_S(v_2) = i_2, \dots, d_S(v_k) = i_k} = \Prob{Z_{d(v_1)} = i_1} \left( 1 + O\left( \bo \left(\frac{\Delta^2}{d(S)} + \frac{\bo^2 M}{d(S)^2} \right) \right) \right),
		\end{align*}
		as well as 
		\begin{align*}
			\Prob{d_S(v_1) = i_1} = \Prob{Z_{d(v_1)} = i_1} \left( 1 + O\left( \bo \left(\frac{\Delta^2}{d(S)} + \frac{\bo^2 M}{d(S)^2} \right) \right) \right).
		\end{align*} 
Let $C_{i}$ be the set of graphs in $\Gnd$ such that $d_S(v_1) = i$ and $d_S(v_j) = i_j$ for all $j \geq 2$. That is, for all $G \in C_i$, $(d_S(v_1), d_S(v_2),  \dots, d_S(v_k)) = (i, i_2, \dots, i_k)$. We apply a switching similar to the one used in \Cref{lem:degree-switching} to switch between $C_{i+1}$ and $C_{i}$. This switching is illustrated in \Cref{fig:switching-sd-pairs-2}. The important difference between this switching and the switching used in the proof of \Cref{lem:degree-switching} is that the induced degrees of vertices $v_2, \dots, v_k$ are maintained. Other than this extra restriction, the edges are chosen in the same way as the switching used in \Cref{lem:degree-switching}. 

		\begin{figure}[h!]
			\begin{center}
				\begin{tikzpicture}[line width=.5pt,vertex/.style={circle,inner sep=0pt,minimum size=0.2cm}, scale=.6]
					
					% % % vertices

					\node [black, fill=black,label={[label distance=0mm]45:$y$}] (y) at (45:2)  [vertex]{}; 
					\node [draw, black, fill=white,label={[label distance=0mm]135:$v_1$}] (v) at (135:2)  [vertex]{}; 
					\node [black, fill=black,label={[label distance=0mm]225:$u$}] (u) at (225:2)  [vertex]{};
					\node [black, fill=black,label={[label distance=0mm]315:$x$}] (x) at (315:2)  [vertex]{}; 
					\node [label={[label distance=0mm]90:$G \in C_{i+1}$}] at (90:2.5) [vertex]{};

					% % % lines
					
					\draw[thick, -] (y) to (v);
					\draw[thick, -] (x) to (u);
					\draw[thick, dashed] (u) to (v);
					\draw[thick, dashed] (y) to (x);
					
					\draw[thick, ->] (2.5,0) to (3.5,0);

				\end{tikzpicture}
				\begin{tikzpicture}[line width=.5pt,vertex/.style={circle,inner sep=0pt,minimum size=0.2cm}, scale=.6]

					% % % vertices

					\node [black, fill=black,label={[label distance=0mm]45:$y$}] (y) at (45:2)  [vertex]{}; 
					\node [draw, black, fill=white,label={[label distance=0mm]135:$v_1$}] (v) at (135:2)  [vertex]{};
					\node [black, fill=black,label={[label distance=0mm]225:$u$}] (u) at (225:2)  [vertex]{};
					\node [black, fill=black,label={[label distance=0mm]315:$x$}] (x) at (315:2)  [vertex]{}; 
					\node [label={[label distance=0mm]90:$G' \in C_{i}$}] at (90:2.5) [vertex]{};

					% % % lines
					
					\draw[thick, dashed] (y) to (v);
					\draw[thick, dashed] (x) to (u);
					\draw[thick, -] (u) to (v);
					\draw[thick, -] (y) to (x);

				\end{tikzpicture}
				
			\end{center}
			\caption{The switching used in this proof. Here $v_1, y \in S$ and $u \in \Sbar$. Importantly, the switching does not alter $d_S(v_2), \dots, d_S(v_k)$.}
			\label{fig:switching-sd-pairs-2}
		\end{figure}
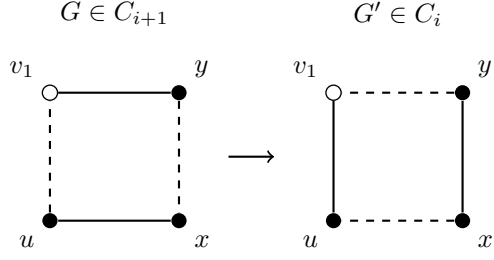
		Now we define the switching formally. Suppose $G \in C_{i+1}$. Choose a vertex $y$ such that $v_1 y \in E(G)$ and $y \in S$, as well as an ordered pair of vertices $(u,x)$ such that $u x \in E(G)$ and $u \in \Sbar$, with the extra restrictions that
		\begin{enumerate}[label=(\alph*)]
					\item the vertices $\left\{u,v_1,x,y\right\}$ are distinct,
					\item $xy \notin E(G)$ and $uv_1 \notin E(G)$,
					\item  the degrees of $v_1,\dots, v_k$ in $G[S]$ are unchanged by the switching. 
		\end{enumerate}
		The switching deletes edges $v_1y$ and $ux$, replacing these edges with $uv_1$ and $xy$ and hence creating a new graph $G^\prime \in C_i$.

		To count the switchings that can be applied to each $G \in C_{i+1}$, we carry out a computation analogous to that in the proof of \Cref{lem:degree-switching}. For each $G \in C_{i+1}$, there are $(i+1) d(\Sbar)$ choices for $\{u,v_1,x,y\}$ such that $v_1 y, ux \in E(G)$, $y \in S$ and $u \notin S$. To estimate the number of switchings, we bound from above the expected number of choices for $\{u,v_1, x, y\}$ such that (a), (b) or 
		\begin{enumerate} 
 		 	\item[(c$'$)] $\{v_2, \dots, v_k\} \cap \{u, v_1, x, y\} = \emptyset$
		\end{enumerate}
is false.
		Note that (c$'$) is   slightly stricter than (c). As before, there are at most $3(i+1)\Delta^2$ choices for $\{u,v_1,x,y\}$ that do not satisfy (a) or (b). Now we consider (c). 
		By assumption, $v_1 \neq v_j$ for $j \geq 2$, and $u \notin S$. Thus, the only possibilities for a non-empty intersection are if $x \in \{v_2, \dots, v_k\}$ or $y \in \{v_2, \dots, v_k\}$. 
		In the former case, there are at most $k\Delta$ choices for a neighbour $u$ of $x$. With at most $i+1$ choices for $y$ adjacent to $v_1$, this means that there are at most  $(i+1)k\Delta$ choices for $\{u,v_1,x,y\}$. 
		In the second case, there are $O(1)$ choices for $y \in \{v_2, \dots, v_k\}$, and for each one \Cref{lem:sd-adjacency-probability} gives
		\begin{align*}
			\CProb{v_1y \in E(G)}{d_S(v_1) = i+1, \dots, d_S(v_k ) = i_k} &= O\left( \frac{\bo^2M}{d(S)^2} \right).
		\end{align*}	
		In such cases, there are $d(\Sbar)$ choices for $ux$.
		Thus, the expected number of choices for $\{u, v_1, x, y\}$ where $y \in \{v_2, \dots, v_k\}$, for a random $G \in C_{i+1}$, is $d(\Sbar)O\left( \frac{\bo^2M}{d(S)^2} \right)$.
		Therefore, the average number of valid switchings that can be applied to each $G \in C_{i+1}$ is 
		\begin{align*}
			(i+1) (d(\Sbar) - O(\Delta^2) ) - d(\Sbar)O\left( \frac{\bo^2M}{d(S)^2} \right) = (i+1) d(\Sbar) \left(1 - O\left( \frac{\Delta^2}{M} + \frac{\bo^2 M}{(i+1)d(S)^2} \right) \right),
		\end{align*}
		since $d(\Sbar) = \Theta(M)$. 

		Next we determine upper and lower bounds for the number of switchings that create a particular $G' \in C_{i}$. 
		There are $d(v_1) - i$ choices for the vertex $u \notin S$ such that $v_1 u \in E(G)$, and $d(S)$ choices for an ordered pair of adjacent vertices $(x,y)$ such that $y \in S$. Each such choice is valid if all of the following conditions are satisfied:
		\begin{enumerate}[label=(\roman*)]
			\item the vertices $\left\{u,v_1,x,y\right\}$ are distinct,
			\item $v_1y \notin E(G')$ and $ux \notin E(G')$,	
			\item the induced degrees of $v_2, \dots, v_k$ are unchanged by the switching.	
		\end{enumerate}
		By the same reasoning as used in the proof of \Cref{lem:degree-switching}, the number of choices for these vertices that do not satisfy one of (i) or (ii) is $O((d(v_1) - i) \Delta^2)$. 
		
		Again, as an upper bound on the number of choices that do not satisfy (iii), we count the choices where $\{u,v_1,x,y\}$ and $\{v_2, \dots, v_k\}$ intersect non-trivially. 		
		Note that $v_1 \neq v_j$ for any $j \geq 2$ by assumption, and $u \neq v_j$ since $u \in \Sbar$. Thus, we only need to consider $x \in \{v_2, \dots, v_k\}$ or $y \in \{v_2, \dots, v_k\}$, which give at most $2(d(v_1) - i) \sum_{j=2}^k d(v_j)$ choices for $\{u,v_1, x,y\}$. Since $k = O(1)$ and each $v_j$ has degree at most $\bo$ (since $\bo$ bounds the maximum degree of the vertices in $\Ssmall$), the number of switchings that can create a given $G' \in C_{i}$ is 
		\begin{align*}
			(d(v_1) - i)d(S) \left( 1 - O\left( \frac{\Delta^2}{d(S)} + \frac{\bo}{d(S)} \right) \right).
		\end{align*} 
		Thus, it follows that
		\begin{align}\label{eqn:conditional-k-ratio}
			\frac{|C_i|}{|C_{i+1}|} = \frac{i+1}{d(v_1) - i}\frac{d(\Sbar)}{d(S)} \left( 1 + O\left( \frac{\Delta^2}{d(S)} + \frac{\bo^2 M}{d(S)^2} \right) \right).
		\end{align}
		Let $p_i = \CProb{d_S(v_1) = i}{d_S(v_2) = i_2, \dots, d_S(v_k) = i_k}$ for $i \in \left\{ 0,\dots, d(v_1) \right\}$. Recall that $\gamma = d(S)/M$. Then 
		\labelcref{eqn:conditional-k-ratio} implies that, for all $i \in \left\{0,\dots, d(v_1) \right\}$,
		\begin{align*}
			\frac{p_{i+1}}{p_{i}} = \frac{d(v_1)-i}{i+1} \frac{d(S)}{d(\Sbar)} \left( 1 + O\left( \frac{\Delta^2}{d(S)} + \frac{\bo^2 M}{d(S)^2} \right) \right) = \frac{d(v_1)-i}{i+1} \frac{\gamma}{1 - \gamma} \left( 1 + O\left( \frac{\Delta^2}{d(S)} + \frac{\bo^2 M}{d(S)^2} \right) \right).
		\end{align*}
		Thus, we can express $p_i$ in terms of $p_0$: 
		\begin{align}\label{eqn:pi-switching}
			p_i &= {d(v_1) \choose i} \left( \frac{1 - \gamma}{\gamma} \right)^i  p_0 \left( 1 + O\left( \frac{\Delta^2}{d(S)} + \frac{\bo^2 M}{d(S)^2} \right) \right)^i\nonumber\\
			&= {d(v_1) \choose i} \left( \frac{1 - \gamma}{\gamma} \right)^i  p_0 \exp \left( O\left( i\left(\frac{\Delta^2}{d(S)} + \frac{\bo^2 M}{d(S)^2} \right)\right)\right)\nonumber\\
			&= {d(v_1) \choose i} \left( \frac{1 - \gamma}{\gamma} \right)^i p_0  \left( 1 + O\left( d(v_1)\left( \frac{\Delta^2}{d(S)} + \frac{\bo^2 M}{d(S)^2} \right) \right) \right),
		\end{align}
		since $i \leq d(v_1)$ and the error term goes to 0 when $d(v_1) \leq \bo$. The sum of all $p_i$ must be equal to 1, and thus
		\begin{align*}
			1 & 
			= \sum_{i=0}^{d(v_1)} \left[ {d(v_1) \choose i} \left( \frac{\gamma}{1-\gamma}\right)^i  \left( 1 + O\left( d(v_1) \left( \frac{\Delta^2}{d(S)} + \frac{\bo^2 M}{d(S)^2} \right) \right) \right) p_0\right].
		\end{align*}
		Since the error is uniformly bounded for all terms in the sum, and all terms are positive, the relative error of the whole sum is at most $O\left( d(v_1) \left( \frac{\Delta^2}{d(S)} + \frac{\bo^2 M}{d(S)^2} \right) \right)$. Thus, 
		\begin{align*}
			\sum_{i=0}^{d(v_1)} \left[ {d(v_1) \choose i} \left( \frac{\gamma}{1-\gamma}\right)^i p_0\right] = 1 + O\left( d(v_1) \left( \frac{\Delta^2}{d(S)} + \frac{\bo^2 M}{d(S)^2} \right) \right). 
		\end{align*}
		It follows from the previous equation that
		\begin{align*}
			p_0 &= \left( \sum_{i=0}^{d(v_1)} {d(v_1) \choose i} \left( \frac{\gamma}{1-\gamma}\right)^i \right)^{-1} \left( 1 + O\left( d(v_1)  \left( \frac{\Delta^2}{d(S)} + \frac{\bo^2 M}{d(S)^2} \right) \right) \right) \\
			&= (1 - \gamma)^{d(v_1)} \left( 1 + O\left( d(v_1)  \left( \frac{\Delta^2}{d(S)} + \frac{\bo^2 M}{d(S)^2} \right) \right) \right).
		\end{align*}
		Applying \Cref{eqn:pi-switching} for all $i \leq d(v_1)$, we obtain that
		\begin{align}\label{eqn:conditional}
			\CProb{d_S(v_1) = i_1}{d_S(v_2) = i_2, \dots, d_S(v_k) = i_k} &= \Prob{Z_{d(v_1)} = i_1} \left( 1 + O\left( d(v_1)  \left( \frac{\Delta^2}{d(S)} + \frac{\bo^2 M}{d(S)^2} \right) \right) \right)
		\end{align}
		for each choice of $i_j \leq d(v_j)$ for all $j \leq k$. By the law of total probability, summing over all ordered tuples $(i_2, \dots, i_k)$ such that $i_j \leq d(v_j)$ for all $j \geq 2$ gives that
		\begin{align}\label{eqn:conditional-sum}
			\Prob{d_S(v_1) = i_1} &= \sum_{(i_2, \dots, i_k)} \prob\left({d_S(v_1) = i_1}\left|{\bigcup_{j=2}^{k} \{d_S(v_j) = i_j\}}\right.\right) \Prob{\bigcup_{j=2}^{k} \{d_S(v_j) = i_j\}}\nonumber\\
			&= \Prob{Z_{d(v_1)} = i_1} \left( 1 + O\left( d(v_1)  \left( \frac{\Delta^2}{d(S)} + \frac{\bo^2 M}{d(S)^2} \right) \right) \right).
		\end{align}
		Since $d(v_1) \leq \bo$, the two required bounds follow.
	\end{proof}

	With this result, we can prove \Cref{lem:dh-ds-comparison}(b). We also use the following Chernoff bound: if $X \sim \Bin(n,p)$ and $\varepsilon > 0$, then 
	\begin{align}\label{eqn:chernoff}
		\Prob{|X - np| \geq \varepsilon np} \leq 2\exp\left( -\varepsilon^2np/3 \right).
	\end{align}
	
	\begin{lemma}\label{lem:low-deg-not-big}
		For all $k$ such that $d(i_k) \leq \delta^{1/32}\co$, the following holds. We have $d_H(k) \leq 2\gamma \delta^{1/32}\co$, and with probability $1 - o(M^{-2})$, we have $d_S(k) \leq 2\gamma \delta^{1/32}\co$.
	\end{lemma}
	\begin{proof}
		Suppose that $k$ is such that $d(i_k) \leq \delta^{1/32}\co$. The first claim is equivalent to the claim that $N(2\gamma\delta^{1/32}\co) \geq k$, where $N(x)$ is defined in \Cref{def:dh} as the sum over all $i \in S$ of the probability that $Z_{d(i)} \leq x$, rounded to the nearest integer. Since $\d$ is ordered in non-decreasing order, we know that $d(i_j) \leq d(i_k) \leq \delta^{1/32}\co$ for all $j \leq k$. Thus, the Chernoff bound given in \labelcref{eqn:chernoff} implies that
		\begin{align}\label{eqn:small-ish}
			\Prob{Z_{d(i_j)} \leq 2\gamma\delta^{1/32}\co} \leq \Prob{Z_{d(i_k)} \leq 2\gamma \delta^{1/32}\co} \leq \Prob{Z_{\delta^{1/32}\co} \leq 2\gamma \delta^{1/32}\co} = 1 - o(M^{-3})
		\end{align}
		for all $j \leq k$. Therefore,
		\begin{align*}
			\sum_{i \in S} \Prob{Z_{d(i)} \leq 2\gamma\delta^{1/32}\co} \geq k - o(M^{-2}),
		\end{align*}
		and thus $N(2\gamma\delta^{1/32}\co) \geq k$. This proves the first claim. For the second claim, note that if $d(k) \leq \delta^{1/32}\co$, then $k \in \Ssmall$. Therefore, \Cref{lem:binomial-degrees} implies that $\Prob{d_S(k) = j} = \Prob{Z_{d(i_k)} = j}(1 + o(1))$. Thus, 
		\begin{align*}
			\Prob{d_S(k) > 2\gamma\delta^{1/32}\co} = \Prob{Z_{d(i_k)} > 2\gamma\delta^{1/32}\co}(1 + o(1)) \leq \Prob{Z_{\delta^{1/32}\co} > 2\gamma\delta^{1/32}\co}(1 + o(1)).
		\end{align*}
		Combining this with \Cref{eqn:small-ish} then implies that $\Prob{d_S(k) > 2\gamma\delta^{1/32}\co} = o(M^{-3})$. The second claim then follows from the union bound over all such $k$. 
	\end{proof}

\subsection{Concentration of the counts of small vertex degrees}

	Recall the definitions of $Y_i$ and $\ytil{i}$:
	\begin{align*}
		Y_i = \sum_{j \in \Ssmall} \ind{d_S(j) = i} \quad \mbox{and} \quad \ytil{i} = \sum_{j \in \Ssmall} \Prob{Z_{d(j)} = i}.
	\end{align*}
	We can show concentration of $Y_i$ around $\ytil{i}$ for all $i \leq \co$. Firstly, \Cref{lem:binomial-degrees} implies that 
	\begin{align*}
		\Exp{Y_i} = \sum_{j \in \Ssmall} \Prob{d_S(j) = i} = \ytil{i}\left( 1 + O\left( \bo \left( \frac{\Delta^2}{d(S)} + \frac{\bo^2 M}{d(S)^2} \right) \right) \right).
	\end{align*}
	We also know that if $\d_S$ satisfies the concentration bound given in \Cref{lem:dh-ds-comparison}(a) then $n_i(\d_S) = Y_i$ for all $i \leq \frac{1}{2}\gamma \co$. Since this is \aas\ true, studying $Y_i$ and $n_i(\d_S)$ are effectively equivalent for small $i$. It also follows that 
	\begin{align*}
		n_i(\d_H) = N(i) - N(i-1) &= \left\lfloor \sum_{i \in S} \Prob{Z_j \leq i} + \frac{1}{2} \right\rfloor - \left\lfloor \sum_{i \in S} \Prob{Z_j \leq i - 1} + \frac{1}{2} \right\rfloor\\
								&= \sum_{i \in S} \Prob{Z_j = i} \pm 1 =  \sum_{i \in \Ssmall} \Prob{Z_j = i} \pm (1 + o(M^{-5})). 
	\end{align*}
	Thus, $n_i(\d_H) = \ytil{i} \pm 2$.	We study $Y_i$ rather than $n_i(\d_S)$ as the analysis is slightly neater.
	We get the following bound on the variance of $Y_i$.
	\begin{lemma}
		For all $i \leq \co$, $\Var{Y_i} \leq \Exp{Y_i}\left( 1 + O\left( \bo \left( \frac{\Delta^2}{d(S)} + \frac{\bo^2 M}{d(S)^2} \right) \right) \Exp{Y_i} \right)$. 
	\end{lemma}
	\begin{proof}
		Let $V_k$ be the indicator variable for the event that vertex $k \in \Ssmall$ has induced degree $i$. Then $Y_i = \sum_{k \in \Ssmall} V_k$.	
		This implies that
		\begin{align}\label{eqn:var-sum}
			\Var{Y_i} &= \Var{\sum_{k \in \Ssmall} V_k} = \sum_{k \in \Ssmall} \Var{V_k} + \sum_{j \neq k} \Cov{V_j, V_k},
		\end{align}
		where the last sum is over all ordered pairs $(j,k) \in \Ssmall^2$ where $j \neq k$.
		Each $V_k$ is a Bernoulli random variable, and thus
		$\Var{V_k} = \Prob{d_S(k) = i} \left( 1 - \Prob{d_S(k) = i} \right)$. 
		So the first summation is bounded by
		\begin{align}\label{eqn:var}
			\sum_{k \in \Ssmall} \Var{V_k} = \sum_{k \in \Ssmall} \Prob{d_S(k) = i}(1 - \Prob{d_S(k) = i}) \leq \Exp{Y_i}.
		\end{align}
		Now consider the covariance summation in \Cref{eqn:var-sum}. 
		An application of \Cref{lem:low-degree-adjacency} provides a bound on the covariance terms in the second summation when $j,k \in \Ssmall$:
		\begin{align*}
			\Cov{V_j, V_k} &= \Prob{d_S(j) = i,\ d_S(k) = i} - \Prob{d_S(j) = i} \Prob{d_S(k) = i}\\
			&= \Prob{d_S(j) = i} \Prob{d_S(k) = i} O\left( \bo \left( \frac{\Delta^2}{d(S)} + \frac{\bo^2 M}{d(S)^2} \right) \right).
		\end{align*}
		Thus, the summation in \Cref{eqn:var-sum} of the covariances over all ordered pairs $(j,k) \in \Ssmall^2$ where $j \neq k$ is equal to
		\begin{align*}
			&\left(\sum_{j \in \Ssmall} \sum_{k \in \Ssmall} \Prob{d_S(j) = i} \Prob{d_S(k) = i}  - \sum_{j \in \Ssmall} \Prob{d_S(j) = i}^2 \right) O\left( \bo \left( \frac{\Delta^2}{d(S)} + \frac{\bo^2 M}{d(S)^2} \right) \right)\\
			&= \left(\left( \sum_{k \in \Ssmall} \Prob{d_S(k) = i} \right)^2 - \sum_{j \in \Ssmall} \Prob{d_S(j) = i}^2 \right) O\left( \bo \left( \frac{\Delta^2}{d(S)} + \frac{\bo^2 M}{d(S)^2} \right) \right).
		\end{align*}
		Noting that $\sum_{j \in \Ssmall} \Prob{d_S(j) = i}^2 \in [0, \Exp{Y_i}^2]$, we obtain
		\begin{align}\label{eqn:cov}
			\sum_{j \neq k \in \Ssmall} \Cov{V_j, V_k} &\leq \Exp{Y_i}^2 \cdot O\left( \bo \left( \frac{\Delta^2}{d(S)} + \frac{\bo^2 M}{d(S)^2} \right) \right).
		\end{align}
		Combining (\ref{eqn:var}) and (\ref{eqn:cov}) gives an upper bound on the variance of $Y_i$:
		\begin{align*}
			\Var{Y_i} &\leq \Exp{Y_i} 
			+ O\left( \bo \left( \frac{\Delta^2}{d(S)} + \frac{\bo^2 M}{d(S)^2} \right) \right) \Exp{Y_i}^2.
		\end{align*}
		The claim of the lemma immediately follows.
	\end{proof}

	\begin{remark}\label{rem:error-bound}
		Recall that we assume $\Delta^2 (\gamma^{-1} \log M)^{12} \leq \delta \gamma M$ for some $\delta \to 0$ slowly. Since we suppose that $\delta = \omega(\log^{-1} M)$, this implies a bound on $\frac{\Delta^2}{d(S)} + \frac{\co^2 M}{d(S)^2}$:
		\begin{align*}
			\frac{\Delta^2}{d(S)} &\leq \delta\frac{1}{\gamma^{-12} \log^{12} M} \leq \delta^{1/2} \gamma \co^{-8} \log^{-1}M,\\
			\frac{\co^2 M}{d(S)^2} &= \frac{\co^2M}{\gamma^2 M^2} \leq \gamma \co^{-8} o(\log^{-2} M)  \leq \delta^{1/2} \gamma \co^{-8} \log^{-1}M,
		\end{align*}
	and in particular, $\frac{\Delta^2}{d(S)} + \frac{\co^2 M}{d(S)^2} = o(\gamma \co^{-8} \log^{-1}M)$. 
	\end{remark}

	\begin{lemma}\label{lem:concentration-big-small}
		Suppose $i \leq \frac{1}{2}\gamma\co$. If $\Exp{Y_i} \geq \gamma^{-1}\co^{7}\sqrt{\log M}$, then
		\begin{align*}
			\Prob{|Y_i - \Exp{Y_i}| \geq \frac{\gamma\Exp{Y_i}}{2\co^{3}}} = O\left( \frac{\gamma^{-1}\co^{-1}}{\sqrt{\log M}} \right),
		\end{align*}
		and if $\Exp{Y_i} \leq \gamma^{-1}\co^{7}\sqrt{\log M}$, then
		\begin{align*}
			\Prob{|Y_i - \Exp{Y_i}| \geq \co^{4}\sqrt{\log M} } = O\left( \frac{\gamma^{-1}\co^{-1}}{\sqrt{\log M}} \right).
		\end{align*} 
	\end{lemma}
	\begin{proof}
		First suppose that $\Exp{Y_i} \geq \gamma^{-1}\co^{7}\sqrt{\log M}$. Applying Chebyshev's inequality with $t  = \alpha \Exp{Y_i}$ gives
		\begin{align*}
			\Prob{ |Y_i - \Exp{Y_i}| \geq \alpha \Exp{Y_i} } &\leq \frac{\Var{Y_i}}{\alpha^2 \Exp{Y_i}^2}
			\leq  \frac{1}{\alpha^2 \Exp{Y_i}} + \frac{1}{\alpha^2} O\left( \bo \left( \frac{\Delta^2}{d(S)} + \frac{\bo^2 M}{d(S)^2} \right) \right).
		\end{align*}
		Now let $\alpha = \frac{1}{2}\gamma\co^{-3}$. Then since $\bo \leq \co$,
		\begin{align*}
			\frac{\Var{Y_i}}{\alpha^2 \Exp{Y_i}^2} &\leq  \frac{4\co^{6}}{\gamma^2\Exp{Y_i}} +  4\gamma^{-2}\co^{7} O\left( \frac{\Delta^2}{d(S)} + \frac{\co^2 M}{d(S)^2} \right)
												\leq 8(\gamma \co \sqrt{\log M})^{-1},
		\end{align*}
		since \Cref{rem:error-bound} states that $\frac{\Delta^2}{d(S)} + \frac{\co^2 M}{d(S)^2} = o(\gamma \co^{-8}\log^{-1} M)$. This proves the first part of the lemma. Now suppose that $\Exp{Y_i} \leq \gamma^{-1}\co^{7}\sqrt{\log M}$. Then it follows that 
		\begin{align*}
			\Var{Y_i} &= \Exp{Y_i} + \Exp{Y_i}^2 O\left( \bo \left( \frac{\Delta^2}{d(S)} +  \frac{\bo^2 M}{d(S)^2} \right) \right)
			\leq \gamma^{-1}\co^{7}\sqrt{\log M} (1 + o(1)).
		\end{align*}
		Then applying Chebyshev's inequality with $t := \co^{4}\sqrt{\log M}$ gives that
		\begin{align*}
			\Prob{|Y_i - \Exp{Y_i}| \geq t} \leq \frac{\gamma^{-1}\co^{7} \sqrt{\log M} (1 + o(1))}{\co^{8}\log M} \leq 2 (\gamma \co \sqrt{\log M})^{-1}.
		\end{align*}
		This completes the proof.
	\end{proof}

	We note here that the above proof is the only point in the paper that we use the assumption that $\delta = \Omega((\log\log M)^{-1})$. If the bound on $\frac{J^2M}{d(S)^2}$ in \Cref{rem:error-bound} is replaced with $\delta^{1/4}\gamma\co^{-8}\log^{-1} M$, then the bound on $\delta$ can be removed entirely. This could be useful if slightly different error bounds on $\d_S$ were required, trading precision in the frequencies of the smaller entries for greater precision in the values of the larger entries.

	Now we prove \Cref{lem:dh-ds-comparison}(c).
	\begin{corollary}\label{lem:low-degree-concentration}
		With probability $1 - o(1)$, 
		\begin{align*}
			\left\lvert n_i(\d_S) - n_i(\d_H) \right\rvert \leq \frac{\gamma n_i(\d_H)}{\co^{3}} + \gamma\co^{5}
		\end{align*}
		for all $i \leq \frac{1}{2}\gamma \co$. 
	\end{corollary}
	\begin{proof}
		First we bound $\Exp{Y_i}$ in terms of $n_i(\d_H)$ using \Cref{lem:binomial-degrees} and \Cref{rem:error-bound} bounding the error term. This implies that for each $i\leq\co$,
		\begin{align*}
			\Exp{Y_i} = \sum_{j \in \Ssmall} \Prob{d_S(j) = i} = \sum_{j \in \Ssmall} \Prob{Z_{d(j)} = i} (1 + o(\gamma\co^{-8})) = (n_i(\d_H) \pm 2) (1 + o(\gamma\co^{-8})),
		\end{align*}
		since $n_i(\d_H) = \ytil{i} \pm 2$ for all such $i$. We also know that, with probability $1 - o(1)$, $Y_i = n_i(\d_S)$ for all $i \leq \frac{1}{2}\gamma \co$. Then we apply \Cref{lem:concentration-big-small} and the triangle inequality to say that \aas\, for all $i \leq \frac{1}{2}\gamma\co$, 
		\begin{align*}
			|n_i(\d_S) - n_i(\d_H)| \leq |Y_i - \ytil{i}| + 2 &\leq |Y_i - \Exp{Y_i}| + |\Exp{Y_i} - \ytil{i}| + 2 \\
			&\leq \frac{\gamma\Exp{Y_i}}{2\co^{3}} + \co^{4}\sqrt{\log M} + o(\ytil{i} \gamma \co^{-8}) + 2 \\
			&= \frac{\gamma\ytil{i}(1 + o(1))}{2\co^{3}} + \co^{4}\sqrt{\log M} + 1  + o(\ytil{i} \gamma \co^{-8}) +2\\
			&\leq \frac{\gamma n_i(\d_H)}{\co^{3}} + \gamma\co^{5},
		\end{align*}
		since $\co = \delta^{-1/16}\gamma^{-1} \log M$. This completes the proof. 
	\end{proof}

	The factor of $\gamma$ in the ``multiplicative'' error term (the $\gamma n_i(\d_H)/\co^3$ term) is particularly useful because it implies that, for $k \leq \frac{1}{2}\gamma \co$, \aas\
	\begin{align*}
		\left|\sum_{i=0}^{k} n_i(\d_H) - n_i(\d_S) \right| \leq \sum_{i=0}^k \left(\frac{\gamma n_i(\d_H)}{\co^{3}} + \gamma\co^{5} \right) \leq \gamma^2 \co^6 + \frac{\gamma|S|}{\co^{3}}.
	\end{align*}
	Since $|S| \leq d(S) = \gamma M$ and $M_H \sim \gamma ^2 M$, this implies that this sum is $O(M_H/\co^{3})$. This is particularly useful in certain proofs. The bounds given in \Cref{lem:concentration-big-small} are not optimal in either the multiplicative or the additive error. An similar proof of \Cref{lem:concentration-big-small} with a different threshold for $\Exp{Y_i}$ can give better error margins in one direction or the other, but this choice strikes a balance that is sufficient for our purposes. With this result, we can now prove \Cref{lem:dh-ds-comparison}.
	
	\begin{proof}[Proof of \Cref{lem:dh-ds-comparison}]
		\Cref{lem:low-degree-concentration} proves part (c), and part (b) follows from \Cref{lem:low-deg-not-big}.
		Now it remains to prove part (a).
		We show in \Cref{cor:high-degree-crossover} that \aas\ $d_S(k) = \gamma d(i_k) \left( 1\pm 8\delta^{1/64}\right)$ uniformly for all $k$ such that $d(i_k) \geq \delta^{1/32}\co$. Now we prove claim (a) by showing that deterministically $d_H(k)  = \gamma d(i_k) \left( 1\pm 3\delta^{1/64}\right)$. For each $i \in S$ such that $d(i) \geq \delta^{1/32}\co$, define
		\begin{align*}
			d(i)^- &= \gamma d(i) \left(1 - 3\delta^{1/64}\right)\\
			\mbox{and} &\\
			d(i)^+ &= \gamma d(i) \left(1 + 3\delta^{1/64}\right).
		\end{align*}
		The Chernoff bound given in \eqref{eqn:chernoff} implies that $\Prob{|Z_{d(j)} - \gamma d(j)| \geq \delta^{1/64}\gamma d(j)} \leq 2\exp{-3\delta^{1/32}\gamma\co} = 2M^{-3}$. Thus, 
		\begin{align}
			\Prob{Z_{d(i)} \leq d(i)^+} - \Prob{Z_{d(i)} < d(i)^-} = 1 - o(M^{-2}).\label{eqn:binomial-big}
		\end{align}
		We use \Cref{eqn:binomial-big} to show that
		\begin{align} \label{eqn:N}
			& N\left( d(i_k)^- -1 \right) \leq k-1 \quad \mbox{and} \quad N\left(d(i_k)^+ \right) \geq k
		\end{align}
		for all $k \geq \ell$. Together these statements imply that for all $i_k$ such that $d(i_k) \geq \delta^{1/32}\co$, the $k$\textsuperscript{th} entry of $\d_H$ is between $d(i_k)^-$ and $d(i_k)^+$. To prove these claims, recall that the degree sequence $\d$ is ordered in non-decreasing order, so $d(i_j) \leq d(i_k)$ for all $j \leq k$ and $d(i_j) \geq d(i_k)$ for all $j \geq k$. 
		This means that, by \Cref{eqn:binomial-big},
		\begin{align*}
			\sum_{j \geq k} \Prob{Z_{d(i_j)} < d(i_k)^{-}} = o(M^{-1}).
		\end{align*}
		This in turn implies that
		\begin{align*}
			N\left( d(i_k)^- -1 \right) &= \left\lfloor \sum_{i \in S} \Prob{Z_{d(i)} < d(i_k)^{-}} + \frac{1}{2} \right\rfloor\\
			&= \left\lfloor \sum_{j < k} \Prob{Z_{d(i_j)} < d(i_k)^{-}} + \sum_{j \geq k} \Prob{Z_{d(i_j)} < d(i_k)^{-}} + \frac{1}{2} \right\rfloor\\
			&\leq \left \lfloor k-1 + o(M^{-1}) + \frac{1}{2} \right\rfloor\\
			&= k-1.
		\end{align*}
		This proves the first statement in \Cref{eqn:N}. An analogous idea proves the second half:
		\begin{align*}
			N\left( d(i_k)^+ \right) &= \left\lfloor \sum_{i \in S} \Prob{Z_{d(i)} \leq d(i_k)^{+}} + \frac{1}{2} \right\rfloor\\
			&= \left\lfloor \sum_{j \leq k} \Prob{Z_{d(i_j)} \leq d(i_k)^{+}} + \sum_{j > k} \Prob{Z_{d(i_j)} \leq d(i_k)^{+}} + \frac{1}{2} \right\rfloor\\
			&\geq \left \lfloor k( 1 - o(M^{-1})) + \frac{1}{2} \right\rfloor\\
			&= k.
		\end{align*}
		The combination of these mean that the $k$\textsuperscript{th} entry of $\d_H$ has a value between $d(i_k)^-$ and $d(i_k)^+$, which proves the part (a) of the lemma. This completes the proof.
	\end{proof}

\section{Basic properties of the induced degree sequence}

In this section we use \Cref{lem:dh-ds-comparison} to prove some basic properties of the sequences $\d_S$ and $\d_H$, such as the sum of their entries (or squares of the entries) and the number of non-zero entries in each sequence. These results are useful for proving a wide variety of properties of $G[S]$, including the threshold for the giant component. 
Throughout this section we implicitly suppose that $\d$ is a sequence of length $n$ with even sum and all entries at least 1 and that $S \subset [n]$ such that $(\d, S)$ satisfy \eqref{eqn:conditions}.

\subsection{Total degree}

	For brevity, we define $M_H := M(\d_H)$ and $M_S := M(\d_S)$. For an arbitrary sequence of non-negative integers $\d$, define $M_2(\d) = \sum_{i=1}^{n(\d)} d(i)^2$. 

	\begin{lemma}\label{lem:m-concentration} 
		$M(\d_H) \sim \gamma^2 M$ always. Furthermore, \aas\ $M(\d_S) \sim M(\d_H)$ and $M_2(\d_S) - M_2(\d_H) = o(M_H)$. 
	\end{lemma}
	\begin{proof}
		Recall $M(\d_H) = \sum_{i \leq \Delta} i n_i(\d_H)$ and that $n_i(\d_H) = N(i) - N(i-1)$. Then 
		\begin{align*}
			M(\d_H) &= \sum_{i \leq \Delta} i(N(i) - N(i-1)) = \sum_{i \leq \Delta} i \left(\sum_{j \in S} \Prob{Z_{d(j)} = i} \pm 1\right)\\
			&= \sum_{j \in S} \Exp{Z_{d(j)}} \pm \Delta^2\\
			&= \gamma^2 M \pm \Delta^2.
		\end{align*}
		Since $\Delta^2 \gamma^{-12} \log^{12} M \leq \delta \gamma M$ by assumption, it follows that $\Delta^2 = o(\gamma^2 M)$. This proves the first claim of the lemma. Now we focus on the second claim. We give the proof that \aas\ $M_2(\d_S) \sim M_2(\d_H)$ and it follows by an identical proof that \aas\ $M_S \sim M_H$, since $d_H(i), d_S(i) \geq 0$. Let $k \in \mathbb{N}$ be the smallest index such that $d_H(i) > 2\delta^{1/32} \gamma \co$ for all $i \geq k$. Similarly, let $k'$ be the analogous quantity for $\d_S$. Then 
		\begin{align*}
			M_2(\d_H) = \sum_{i=1}^{\frac{1}{2}\gamma \co} i^2n_i(\d_H) + \sum_{i=k+1}^{s} d_H(i)^2.
		\end{align*}
		Now we apply \Cref{lem:dh-ds-comparison}. Note that $2\delta^{1/32} \gamma \co = o(\gamma \co)$. Without loss of generality, suppose that $k \leq k'$. Then \aas\
		\begin{align*}
			M_2(\d_H) &= \sum_{i=1}^{2\delta^{1/32} \gamma \co} i^2n_i(\d_H) + \sum_{i=k+1}^{s} d_H(i)^2\\
			&= \sum_{i=1}^{2\delta^{1/32} \gamma \co} i^2n_i(\d_S) + O\left( \frac{\gamma M_H}{\co} \right) + o(\gamma\co^6) + \sum_{i=k+1}^{s} d_S(i)^2(1 + o(1))\\
			&= \sum_{i=1}^{2\delta^{1/32} \gamma \co} i^2n_i(\d_S) + \sum_{i=k+1}^{s} d_S(i)^2(1 + o(1)) + o(M_H)\\
			&= M_2(\d_S) + \sum_{i=k+1}^{k'} d_S(i)^2(1+o(1)) + o(M_H).
		\end{align*}
		\Cref{lem:dh-ds-comparison}(c) implies that \aas\ $|k-k'| \leq \frac{\gamma|S|}{\co^{3}} + \delta^{1/32}\gamma\co^6 $, and $d_S(i) \leq 2\delta^{1/32} \gamma \co$ for all $i \leq k'$ by definition. Therefore, \aas\ 
		\begin{align*}
			\sum_{i=k+1}^{k'} d_S(i)^2 \leq (2\delta^{1/32}\gamma\co)^2 \left(\frac{\gamma|S|}{\co^{3}} + \delta^{1/32}\gamma\co^6\right) = o(M_H),
		\end{align*} 
		since $\gamma^3 \co^8 = o(\gamma^2 M/\co)$. Therefore, \aas\ $M_2(\d_H) = M_2(\d_S)(1 + o(1))$, which completes the proof.
	\end{proof}

	\subsection{Concentration of number of non-isolated vertices in $G[S]$}
	
	For a sequence $\d$ of non-negative integers, let $\d^*$ be the maximal subsequence of $\d$ with all positive entries (that is, all entries equal to zero removed).
	
	\begin{lemma}\label{lem:degree-0-vertices}
		$n(\d_H^*) = \Omega(\gamma |S|)$ always, and \aas\ $n(\d_S^*) \sim n(\d_H^*)$. 
	\end{lemma}
	\begin{proof} 
		For the first result, recall that $n_k(\d_H) = \ytil{k} \pm 2$ for all $k \leq \frac{1}{2} \gamma \co$. Then
		\begin{align*}
			\ytil{1} = \sum_{i \in \Ssmall} \Prob{Z_{d(i)} = 1} = \frac{\gamma}{1-\gamma} \sum_{i \in \Ssmall} d(i) \Prob{Z_{d(i)} = 0}.
		\end{align*}
		Since $\d$ has a minimum entry of at least 1, this implies that $\ytil{1} = \Omega(\gamma\ytil{0})$, and the first claim follows since $\ytil{0} \leq |S|$. For the second claim, we apply the first part of this lemma as well as \Cref{lem:dh-ds-comparison}(c). We know that \aas\ $|n_0(\d_H) - n_0(\d_S)| \leq \frac{\gamma n_0(\d_H)}{\co^{3}} + \co^5$. Thus, \aas\
		\begin{align*}
			|n(\d_S^*) - n(\d_H^*)| = |n_0(\d_S) - n_0(\d_H)| \leq \frac{\gamma n_0(\d_H)}{\co^{3}} + \gamma\co^5.
		\end{align*}
		By the definition of $\co$ and our assumptions on $\d$ and $S$, it follows that $\co^5 = o(\gamma|S|)$ (since $|S| \geq d(S)/\Delta$). We also know that $\frac{\gamma n_0(\d_H)}{\co^{3}} = o(\gamma|S|)$, since $n_0(\d_H) \leq n(\d_H) = |S|$. The first claim of this lemma states that $n(\d_H^*) = \Omega(\gamma|S|)$. Therefore, \aas\ $|n(\d_S^*) - n(\d_H^*)| = o(n(\d_H^*))$, which proves the second claim. 
	\end{proof}

\section{Giant components in $G[S]$}

	In this section we prove \Cref{thm:main-gc-pretty}. This is done by applying the following theorem of Joos, Perarnau, Rautenbach, and Reed to the characterisation of the degree sequence of the induced subgraph given in \Cref{lem:dh-ds-comparison}.

	\begin{theorem}\label{thm:joos}(\cite{Joos2018}, Theorems 1 and 6)
		Let $\d = (d(1),\dots, d(n))$ with $d(1) \leq d(2) \leq \dots \leq d(n)$. Define the following quantities:
		\begin{align*}
			j_\d &= \min \left(\left\{ j:j\in \left[n\right] \ \mbox{and} \  \sum_{i=1}^{j} d(i)(d(i)-2) > 0 \right\} \cup \left\{n\right\}\right),\\
			R (\d) &= \sum_{i = j_\d}^n d(i),\\
			\MJ(\d) &= \sum_{i \in [n], d(i) \neq 2} d(i).
		\end{align*}
		Call a degree sequence \emph{well-behaved} if $\MJ(\d)$ is at least $\lambda(n)$ for any function $\lambda: \mathbb{N} \to \mathbb{N}$ where $\lambda \to \infty$ as $n \to \infty$. Then:
		\begin{enumerate}[label=(\alph*)]
			\item For every function $\delta \to 0$ as $n \to \infty$, for every $\gamma > 0$, if $\d$ is a well-behaved graphical sequence with $R(\d) \leq \delta(n) \MJ(\d)$, then the probability that $G(\d)$ has a component of order at least $\gamma n$	is $o(1)$.
			\item For every positive constant $\varepsilon$, there is a $\gamma > 0$ such that if $\d$ is a well-behaved graphical sequence with $R(\d) \geq \varepsilon \MJ(\d)$, then the probability that $G(\d)$ has a component of order at least $\gamma n$ and a component of size at least $\alpha \MJ(\d)$ is $1 - o(1)$. 
			\item For every $b \geq 0$ and every $0 < \gamma < \frac{1}{8}$, there exist a positive integer $n_{b,\gamma}$ and a $0 < \delta < 1$ such that if $n > n_{b,\gamma}$ and $\d$ is a degree sequence with $\MJ(\d) \leq b$, then the probability that there is a component of order at least $\gamma n$ in $G(\d)$ lies between $\delta$ and $1-\delta$. 
		\end{enumerate}
	\end{theorem}

	We prove this theorem by showing that if $\d_S$ satisfies the bounds given in \Cref{lem:dh-ds-comparison} (equivalently, if $\d_S \in \goodds$), then the corresponding values of $R(\cdot)$ and $\MJ(\cdot)$ are close for each sequence. \Cref{lem:degree-0-vertices} states that for each such sequence $\d_S$, the number of non-zero entries is close to the number of non-zero entries in $\d_H$. With all of these results, the proof follows by applying \Cref{thm:joos} and the law of total probability.

	\begin{lemma}\label{lem:joos-m-concentration}
		$\MJ(\d_H) = \Theta(M_H)$, and \aas\ $\MJ(\d_H) - \MJ(\d_S) = o(M_H)$.
	\end{lemma}
	\begin{proof}
		The first claim follows unless $M_H \sim 2n_2(\d_H)$; we prove that the definition of $\d_H$ means that this does not occur. To do this, we show that $\Prob{Z_{d(i)} = 2} \leq C\left( \Prob{Z_{d(i)} = 1} + \Prob{Z_{d(i)} = 3} \right)$ for some constant $C > 0$. Firstly, if $d(i) =1$, then this is trivially true. Next, if $d(i) = 2$, then there exists a constant $C > 0$ such that
		\begin{align*}
			\Prob{Z_{d(i)} = 1} = 2\gamma(1-\gamma) \leq C\gamma^2 = C\Prob{Z_{d(i)} = 2}.
		\end{align*}		
		Finally, if $d(i) \geq 3$, then
		\begin{align*}
			\frac{\Prob{Z_{d(i)} = 2}}{\Prob{Z_{d(i)} = 1} + \Prob{Z_{d(i)} = 3}} &= \frac{\frac{1}{2}(d(i) - 1) \gamma(1-\gamma)}{(1-\gamma)^2 + \frac{1}{6}(d(i) - 1)(d(i) - 2)\gamma^2}.
		\end{align*}
		The expression on the right hand side is decreasing in $d(i)$ if $\gamma d(i) \geq 6$, and is at most $\frac{3}{1-\gamma}$ if $\gamma d(i) \leq 6$. Thus, define $C_\gamma = \frac{3}{1-\gamma}$. By our assumptions on $\gamma$, we know that $C_\gamma = O(1)$ and 
		\begin{align*}
			\Prob{Z_{d(i)} = 2} \leq C_\gamma \left( \Prob{Z_{d(i)} = 1} + \Prob{Z_{d(i)} = 3} \right).
		\end{align*}
		Since $n_i(\d_H) = \sum_{i \in S} \Prob{Z_{d(i)} = i} \pm 2$, this implies that
		\begin{align*}
			\MJ(\d_H) \geq n_1(\d_H) + 3n_3(\d_H) \geq \sum_{i \in S} \left( \Prob{Z_{d(i)} = 1} + 3\Prob{Z_{d(i)} = 3} \right) - 4 \geq C_\gamma^{-1} n_2(\d_H) + O(C_\gamma^{-1}).
		\end{align*}
		This implies that $\MJ(\d_H) = \Omega(n_2(\d_H))$ and thus $M_H \geq (2 + \varepsilon) n_2(\d_H)$ for some $\varepsilon > 0$.
		Therefore, $\MJ(\d_H) = \Theta(M_H)$, which proves the first claim of the lemma. 
		
		For the second claim, \Cref{lem:dh-ds-comparison} implies that \aas\ $|n_2(\d_H) - n_2(\d_S)| \leq \frac{\gamma n_2(\d_H)}{\co^{3}} + \co^{5}$. Since $n_2(\d_H) \leq \frac{1}{2}M_H$ and $\co^5 = o(\gamma^2 M)$, this implies that \aas\ $2n_2(\d_S) = 2n_2(\d_H) + o(M_H)$. Recall from \Cref{lem:m-concentration} that \aas\ $M_S \sim M_H$. Therefore, \aas\
		\begin{align*}
			\MJ(\d_S) - \MJ(\d_H) = M_S - M_H - 2n_2(\d_S) + 2n_2(\d_H) = o(M_H).
		\end{align*}
		This completes the proof. 
	\end{proof}
	\begin{lemma}\label{lem:ksum-difference}
		With probability $1 - o(1)$, it holds that $\sum_{i=1}^{k} \left(d_S(i) - d_H(i)\right) = o(M_H)$ for every $k \leq s$ and $\sum_{i=1}^{k} \left(d_S(i)^2 - d_H(i)^2\right) = o(M_H)$	for every $k \leq s$.
	\end{lemma}
	\begin{proof}
		We provide the proof of the second claim, as the first claim follows immediately from this and the fact that $d_S(i), d_H(i) \geq 0$ for all $i \leq s$. 
		We show that deterministically the result holds when $\d_S$ satisfies the bounds given in \Cref{lem:dh-ds-comparison}.
		Since these concentration results hold with probability $1 - o(1)$, this is sufficient to prove the claim. First consider the case where $k$ is such that $d(i_k) > \delta^{1/32}\co$. In this case, we know that \aas\ $d_S(j) \sim d_H(j)$ for all $j \geq k$, and thus by \Cref{lem:m-concentration} it follows that \aas\
		\begin{align*}
			\sum_{i=1}^k d_S(i)^2 &= M_2(\d_S) - \sum_{i=k+1}^{s} d_S(i)^2 = M_2(\d_H) + o(M_H) - \sum_{i=k+1}^{s} d_H(i)^2(1 + o(1)) \\
			&= M_2(\d_H) - \sum_{i=k+1}^{s} d_H(i)^2 + o(M_H) =  \sum_{i=1}^k d_H(i)^2 + o(M_H).
		\end{align*}
		Thus, it only remains to consider the case where $k$ is such that $d(i_k) \leq \delta^{1/32}\co$. Since we assume $\d_S$ satisfies the bounds given in \Cref{lem:dh-ds-comparison}, it follows that $d_S(i), d_H(i) \leq \delta^{1/32}\co$ for all $i \leq k$. Consider that if $d_S(i) = d_H(i)$, then the contribution of this term is $0$. Thus,
		\begin{align*}
			\sum_{i=1}^{k} \left(d_S(i)^2 - d_H(i)^2\right) = \sum_{i\leq k:d_S(i) \neq d_H(i)} \left(d_S(i)^2 - d_H(i)^2\right).
		\end{align*}
		The difference $d_S(j)^2 - d_H(j)^2$ is bounded by $\delta^{1/16}\co^2$ for all $j \leq k$ by assumption. It remains to bound the size of $\{i \leq k : d_S(i) \neq d_H(i)\}$. Note that for all $j \leq \delta^{1/32}\co$, the difference between the index of the first entry equal to $j$ in $\d_S$ and the first entry equal to $j$ in $\d_H$ is at most 
		\begin{align*}
			\sum_{i=0}^{j-1} \left( n_i(\d_S) - n_i(\d_H) \right).
		\end{align*}
		Since $d_S(i), d_H(i) \leq \delta^{1/32}\co$, \Cref{lem:dh-ds-comparison}(c) implies that
		\begin{align*}
			|\{i\leq k:d_S(i) \neq d_H(i)\}| &\leq \sum_{x=1}^{\delta^{1/32}\co} \sum_{i=0}^{x-1} \left( \frac{\gamma n_i(\d_H)}{\co^{3}} + \gamma\co^{5} \right)\\
			&\leq \delta^{1/16} \gamma\co^{7} + \delta^{1/32}\frac{\gamma|S|}{\co^{2}}.
		\end{align*}
		Since $|S| \leq d(S) = \gamma M$ and $M_H \sim \gamma^2 M$, this implies that
		\begin{align*}
			\sum_{i\leq k:d_S(i) \neq d_H(i)} \left(d_S(i)^2 - d_H(i)^2\right) \leq (\delta^{1/32}\co)^2 \left( \delta^{1/16} \gamma\co^{7} + \delta^{1/32}\frac{\gamma|S|}{\co^{2}} \right) \leq \delta^{1/8}\gamma\co^9 + \delta^{1/16}M_H.
		\end{align*}
		Since $\co^9 = o(\gamma^2 M/\co)$ and $\delta \to 0$, this completes the proof.
	\end{proof}

	\begin{lemma}\label{lem:rd-concentration}
		With probability $1 - o(1)$, $R(\d_S) - R(\d_H) = o(M_H)$. 
	\end{lemma}
	\begin{proof}
		Again, we show that the claim holds deterministically when $\d_S$ satisfies the bounds given in \Cref{lem:dh-ds-comparison}. Since these concentration results hold with probability $1 - o(1)$, this is sufficient to prove the claim. Define $j_H := j(\d_H)$ and $j_S := j(\d_S)$. Without loss of generality, we assume that $j_H \leq j_S$, as the proof is symmetric in the converse case. We also assume that $j_H < n$, otherwise we know that $R(\d_S),\ R(\d_H) \leq \Delta$ and the result is immediate. These assumptions imply that $\sum_{i = 1}^{j_H} d_H(i) (d_H(i) - 2) > 0$. Thus, \Cref{lem:ksum-difference} implies that
		\begin{align*}
			\sum_{i=1}^{j_H} d_S(i) (d_S(i) - 2) \geq -\alpha M_H
		\end{align*}
		for some $\alpha\to 0$. We can choose $\alpha$ such that $\alpha = \omega(\delta^{1/16})$. We also know that the number of entries in each sequence with value $0$, $1$, or $2$ differs by at most $3\co^5 + \sum_{i=0}^2 \frac{\gamma n_i(\d_H)}{\co^{3}}$. We also know that $d_H(j_H) \geq 3$, as otherwise $j_H = n$ by definition. Therefore, we know that there are at most $3\co^5 + \sum_{i=0}^2 \frac{\gamma n_i(\d_H)}{\co^{3}}$ entries in $\d_S$ with value at most $2$ and index at least $j_H$. Choose $j^*$ to be the smallest integer such that $\sum_{i=j_H}^{j^*} d_S(i) \geq \alpha^{1/2} M_H$. Then
		\begin{align*}
			\sum_{i=1}^{j^*} d_S(i) (d_S(i) - 2) &\geq - \left(\alpha M_H + 3\co^5 + \sum_{i=0}^2 \frac{\gamma n_i(\d_H)}{\co^{3}} \right) + \alpha^{1/2}M_H\\
			&\geq \frac{1}{2}\alpha^{1/2}M_H.
		\end{align*}
		Thus, $j_S \leq j^*$. Therefore, $R(\d_S) \geq \sum_{i=j^*}^{s} d_S(i)$. Then \Cref{lem:m-concentration,lem:ksum-difference} imply that 
		\begin{align*}
			\sum_{i=j_S}^{s} d_S(i) &= M_S - \sum_{i=1}^{j_S-1} d_S(i) = M_S - \sum_{i=1}^{j_H-1} d_S(i) - \sum_{i=j_H}^{j_S - 1} d_S(i)\\
			&= M_H - \sum_{i=1}^{j_H-1} d_H(i) - \sum_{i=j_H}^{j_S - 1} d_S(i) + o(M_H)\\
			&= R(\d_H) - o(M_H),
		\end{align*}
		since $j_S \leq j^*$ and $\sum_{i=j_H}^{j^*} d_S(i) \geq \alpha^{1/2} M_H$ by definition. This completes the proof. 		
	\end{proof}

	Now we prove \Cref{thm:main-gc-pretty}. Let $\Ds$ be the set of all possible degree sequences of $G[S]$, and let $\goodds \subset \Ds$ be the set containing all the sequences that satisfy \Cref{lem:dh-ds-comparison}. Implicitly, $\goodds$ depends on the choice of $\delta$. For convenience and brevity we omit this dependence. 

	\begin{proof}[Proof of \Cref{thm:main-gc-pretty}]
		\Cref{lem:degree-0-vertices} implies that $n(\d_S^*)$, the number of non-isolated vertices in $G[S]$, is \aas\ $n(\d_H^*)(1 + o(1)) = (|S| - n_0(\d_H))(1 + o(1))$.
		\Cref{lem:m-concentration,lem:joos-m-concentration,lem:rd-concentration} imply that for all $\d_S \in \mathcal{D}_S'$, we have that 
		\begin{align*}
			R(\d_S) - R(\d_H) = o(\MJ(\d_H)), \quad \MJ(\d_S) \sim \MJ(\d_H).
		\end{align*}
		Since $\MJ(\d_H) = \Theta(M(\d_H))$, this implies that $\d_H$ is well-behaved, and thus $\d_S$ is well-behaved for all $\d_S \in \goodds$. Now suppose that $R(\d_H) \geq \varepsilon \MJ(\d_H)$ for some $\varepsilon > 0$. Then it immediately follows that $R(\d_S) \geq \frac{1}{2}\varepsilon \MJ(\d_S)$ for all $\d_S \in \goodds$. Then \Cref{thm:joos}(b) implies that there exists some $\gamma := \gamma\left( \frac{1}{2}\varepsilon \right) > 0$ such that the probability that $G[S]$ contains a component with at least $\gamma n(\d_S^*)$ vertices is $1 - o(1)$. Let $A$ be the event that $G[S]$ contains a component with at least $\frac{1}{2}\gamma (|S| - n_0(\d_H))$ vertices. Then the law of total probability gives that
		\begin{align*}
			\Prob{A} = \sum_{\d_S \in \Ds} \CProb{A}{\d_S} \Prob{\d_S} = \sum_{\d_S \in \goodds} \CProb{A}{\d_S} \Prob{\d_S} + o(1) = 1 - o(1).
		\end{align*}
		The argument for the converse case is very similar. Suppose that $R(\d_H) \leq \delta'\MJ(\d_H)$ for some $\delta' \to 0$. Then immediately it follows that $R(\d_S) \leq 2\delta' \MJ(\d_S)$ for all $\d_S \in \goodds$. Then \Cref{thm:joos}(a) and the law of total probability imply that for every $\gamma > 0$, the probability that $G[S]$ contains a component with at least $\gamma (|S| - n_0(\d_H))$ vertices is $o(1)$. Since $\d_S \in \goodds$ \aas\, this completes the proof.
	\end{proof}

\section{Site percolation: random induced subgraphs}\label{sec:perc}

\begin{proof}[Proof of \Cref{lem:degree-sequence-concentration-perc}]
	Immediately we know that $\Exp{|S|} = np$, and linearity of expectation gives that
	\begin{align*}
		\Exp{d(S)} = \sum_{i \in [n]} d(i) \Prob{i \in S} = p M.
	\end{align*}
	First we argue concentration of $|S|$. The Chernoff bound given in \eqref{eqn:chernoff} implies that 
	\begin{align*}
		\Prob{| |S| - \Exp{|S|}| \geq \varepsilon np} \leq 2 \exp\left( - \frac{np\varepsilon^2}{3} \right).
	\end{align*}
	Letting $\varepsilon = 3\sqrt{\log n / pn}$, it immediately follows that this probability is at most $2n^{-3}$. This completes the proof of (a). For part (b), we construct a martingale to show concentration of $d(S)$. At step $i$, for $i \in [n]$, reveal whether vertex $i$ is in $S$. 
	Define $M_i = d(S\cap[i])$. Then it follows immediately that $|M_i - M_{i-1}| \leq d(i)$ for all $i \leq n$ and $M_n = d(S)$. Thus, Azuma's inequality implies that
	\begin{align*}
		\Prob{|d(S) - pM| \geq \alpha} \leq 2\exp\left( \frac{-\alpha^2}{2 \sum_{i\in V(G)}d(i)^2} \right).
	\end{align*}
	By assumption, $\Delta(\d) = o\left(p^6 \frac{\sqrt{M}}{\log^6 M} \right)$, which implies that
	\begin{align*}
		\sum_{i\in V(G)}d(i)^2 \leq \Delta(\d) \sum_{i \in [n]} d(i) = o\left( p^6\frac{M^{3/2}}{\log^6 M} \right).
	\end{align*} 
	Thus, by setting $\alpha = p^3M^{3/4}$, this probability is $M^{-\omega(1)}$, which proves that $S$ \aas\ satisfies condition (b). For the remainder of this proof, we call a set $S$ ``good'' if $|S| = pn\left( 1 \pm 3\sqrt{\log n}/\sqrt{pn} \right)$ and $d(S) = pM \left( 1 \pm p^2M^{-1/4} \right)$. It follows that $S$ is good with probability at least $1 - 3n^{-3}$. 
	
	Now we focus on part (c). Let $S$ be some arbitrary ``good'' subset of $[n]$, and recall that $\gamma(S) = d(S)/M$. Then \Cref{lem:dh-ds-comparison}(a) applies and thus it follows that 
	\begin{align*}
		\CProb{d_S(v) = \gamma(S) d(v)\left( 1 \pm  8\delta^{1/64} \right)\ \mbox{for all $v \in S$ such that $d(v) > \cop[\gamma(S)]$}}{S} = 1 - o(1).
	\end{align*}
	By parts (a) and (b) of this lemma, we know that $S$ is good with probability $1 - o(1)$. If $S$ is good, then $\gamma(S) = p(1 \pm p^2M^{-1/4})$. Therefore, $\cop[\gamma(S)] < 2\cop$ for every good set $S$. Thus, the probability that there exists some $v \in S$ with $d(v) > 2\delta^{1/32}\cop$ with induced degree that is not $p d(v) \left( 1 \pm 8 \delta^{1/64} \right)(1 \pm p^2M^{-1/4})$ is $o(1)$. Since $\delta = \omega(M^{-1/4})$, it follows that $(1 \pm 8\delta^{1/64})(1 \pm p^2M^{-1/4}) = 1 \pm 9\delta^{1/64}$.	This proves part (c).
	
	Finally, we focus on part (d) of the lemma. Recall that $n_j(\d)$ is the number of entries of $\d$ equal to $j$, or equivalently (if $\d$ is graphical) the number of vertices with degree $j$ in a graph with degree sequence $\d$. Let $S_j$ be the set of all degree $j$ vertices in $S$. We can express $|S_j|$ as a sum over all vertices in $G$ with degree $i$:
	\begin{align*}
		|S_j| &= \sum_{d(i) = j} \ind{i \in S},
	\end{align*}
	where $\ind{i \in S}$ is an independent Bernoulli random variable with $\Prob{\ind{i \in S} = 1} = p$ for all $i \in [n]$. 
	It follows from linearity of expectation that $\Exp{|S_j|} = p n_j(\d)$. Since these indicators are independent random variables, the Chernoff bound given in \eqref{eqn:chernoff} implies that
	\begin{align*}
		\Prob{||S_j| - p n_j(\d)| > \alpha_j p n_j(\d)} \leq 2\exp\left( -\frac{1}{3}\alpha_j^2 p n_j(\d) \right)
	\end{align*}
	for each $j \leq \cop$. 
	Define $\alpha_j = \left( \frac{6\log \cop}{p n_j(\d)}\right)^{1/2}$. Then 
	\begin{align*}
		\exp\left( -\frac{2\alpha_j^2}{n_j(\d)} \right) = \exp\left( -2\log\cop\right).
	\end{align*}
	Setting $\alpha_j = (6p n_j(\d)\log\cop)^{1/2}$ and performing the union bound over $\cop$ such events implies that the probability that each $S_j$ is within of its expectation is at most $\cop^{-2}$. 
	Suppose for some particular $j$ that $pn_j(\d) \geq p^{-2}\cop^{6}\log^2 M$. Then, with probability $1 - o(1)$ it follows that
	\begin{align*}
		||S_j| - p n_j(\d)| \leq (6 p n_j(\d)\log\cop)^{1/2} \leq (6\log \cop)^{1/2} \frac{pn_j(\d)}{p^{-1}\cop^{3}\log M}.
	\end{align*}
	Now suppose that $pn_j(\d) < p^{-2}\cop^{6}\log^2 M$. Then it follows that 
	\begin{align*}
		||S_j| - p n_j(\d)| \leq (6p n_j(\d)\log\cop)^{1/2} < (6\log\cop)^{1/2} p^{-1}\cop^{3}\log M.
	\end{align*}
	Noting that $\log \cop \leq \frac{1}{6}\log M$, it follows that \aas\ for all $j \leq \cop$
	\begin{align}\label{eqn:sj-concentration}
		||S_j| - p n_j(\d)| \leq  \frac{p^2n_j(\d)}{\cop^{3}\sqrt{\log M}} + \delta^{1/16}\cop^{4}\sqrt{\log M}.
	\end{align}
	Suppose that $S$ is an arbitrary good subset of $[n]$ which also satisfies the concentration bounds given in \Cref{eqn:sj-concentration} for all $j \in \{0, \dots, \co\}$. Since $S$ is good, it follows that $\gamma(S) = p(1 \pm p^2 M^{-1/4})$. The set $S$ being good also implies that the conditions of \Cref{lem:binomial-degrees} are met, and also that $\frac{1}{3}p\cop < \frac{1}{2}\gamma(S)\cop[\gamma(S)]$. Recall from \Cref{def:dh} the definition of $n_i(\d_H(S))$, the number of entries in $\d_H$ (for a given set $S$) with value $i$. Since $i \leq \frac{1}{2}\gamma(S)\cop[\gamma(S)]$, 
	\begin{align*}
		n_i(\d_H(S)) &= \sum_{j \leq \cop[\gamma(S)]} |S_j| \Prob{Z_j(S) = i} \pm \left( 1 +  o(M^{-5}) \right),
	\end{align*}
	where $Z_j(S) \sim \Bin{j, \gamma(S)}$.
	Suppose $i \leq \frac{1}{3}p\cop = \frac{1}{3} \gamma(S) \cop[\gamma(S)]$. Since $S$ is good, the Chernoff bound \eqref{eqn:chernoff} implies that $\Prob{Z_j(S) = i} = o(M^{-5})$ for all $j \geq \frac{1}{2}\cop$. Thus, for $i \leq \frac{1}{3}p\cop$, 
	\begin{align}\label{eqn:ndhi}
		n_i(\d_H(S)) &= \sum_{j \leq \cop} |S_j| \Prob{Z_j(S) = i} \pm  2.
	\end{align}
	Now we compare the summation to the value of $\tilde{w}_i$. If $d(v) \leq \cop$, then for all $i \leq d(v)$ it follows that
	\begin{align*}
		\frac{\Prob{X_{d(v)} = i}}{\Prob{Z_{d(v)} = i}} &= \left( \frac{p}{\gamma(S)} \right)^i \left( \frac{1-p}{1-\gamma(S)} \right)^{d(v) - i}\\
				&= \left( 1 + O\left( \frac{p^2 i}{M^{1/4}} \right) \right) \left( 1 + O\left( \frac{p^3 (d(v) - i)}{M^{1/4}} \right) \right)\\
				&= 1 + O\left( \frac{p^2 \cop}{M^{1/4}} \right).
	\end{align*}
	Recall that the conditions on $p$ (specifically \Cref{eqn:conditions-perc}) imply that $p = \omega(M^{-1/13})$. This implies that $\frac{p^2 \cop}{M^{1/4}} \leq \frac{p}{M^{1/100} \cop^3}$.
	Therefore, 
	\begin{align}\label{eqn:zj-xj}
		\sum_{j \leq \cop} |S_j| \Prob{Z_j(S) = i} = \sum_{j \leq \cop} |S_j| \Prob{X_j = i} \left( 1 + o \left( \frac{p}{\cop^3} \right) \right).
	\end{align}
	Since $i \leq \frac{1}{3}p\cop$, the Chernoff bound given in \eqref{eqn:chernoff} implies that
	\begin{align*}
		\tilde{w}_i = p\sum_{v \in V} \Prob{X_{d(v)} = i} = p\sum_{j \leq \cop} n_j(\d) \Prob{X_j = i} + o(1).
	\end{align*}	
	Since we assume that $S$ satisfies the concentration inequalities given in \Cref{eqn:sj-concentration}, it follows that
	\begin{align}\label{eqn:xj-wi}
		\sum_{j \leq \cop} |S_j| \Prob{X_j = i} &= \sum_{j \leq \cop} \Prob{X_j = i} \left(p n_j(\d) + o\left( \frac{p^2n_j(\d)}{\cop^{3} } + \cop^4\sqrt{\log M} \right) \right)\nonumber\\
		&= \sum_{j \leq \cop} \Prob{X_j = i} p n_j(\d) \left( 1 + o\left( \frac{p}{\cop^{3}} \right) \right) +  \sum_{j \leq \cop} \Prob{X_j = i} o\left(\frac{p\cop^{5}}{\sqrt{\log M}}\right)\nonumber \\
		&= \tilde{w}_{i} + o\left( \frac{p\tilde{w}_{i}}{\cop^{3}} + \frac{p\cop^{6}}{\sqrt{\log M}} \right).
	\end{align}
	Combining \Cref{eqn:ndhi,eqn:zj-xj,eqn:xj-wi} it follows that, conditional on the aforementioned good set $S$,
	\begin{align}\label{eqn:condexp-yi}
		n_i(\d_H(S)) = \tilde{w}_{i} + o\left( \frac{p\tilde{w}_{i}}{\cop^{3}} + \frac{p\cop^{6}}{\sqrt{\log M}} \right). 
	\end{align}
	Since $S$ is good, the pair $(\d, S)$ also satisfy the conditions of \Cref{lem:dh-ds-comparison}(c). This implies that, for this fixed set $S$, \aas\
	\begin{align}\label{eqn:condexp-Yi}
		|n_i(\d_S) - n_i(\d_H(S))| \leq \frac{\gamma(S) n_i(\d_H(S))}{\cop[\gamma(S)]^{3}} + \gamma(S)\cop[\gamma(S)]^5.
	\end{align}
	Recall that $n_i(\d_A) =\tilde{w}_i \pm 1$ for all $i$. Thus, the bounds given in \labelcref{eqn:condexp-yi,eqn:condexp-Yi} and the triangle inequality imply that, conditional on the event that $S$ is good and also satisfies the concentration bounds given in \Cref{eqn:sj-concentration}, \aas\ 
	\begin{align*}
		|n_i(\d_S) - n_i(\d_A)| \leq |n_i(\d_S) - n_i(\d_H(S))| + |n_i(\d_H(S)) - n_i(\d_A)| \leq \frac{pn_i(\d_A)}{\cop^{3}}(1 + o(1)) + o\left( \frac{p\cop^{6}}{\sqrt{\log M}} \right)
	\end{align*}
	for all $i \leq \frac{1}{3} p\cop$. Since $S$ satisfies these conditions \aas, this proves part (d). 
\end{proof}
Recall the definition of $M_2(\d) = \sum_{i=1}^n d(i)^2$ where $d$ is a sequence of length $n$. Very similarly to the case where $S$ is fixed, it follows that $M(\d_S)$ and $M_2(\d_S)$ are both concentrated around their corresponding values for $\d_A$. We do not give the full proof here, as it is practically identical the proof of \Cref{lem:m-concentration}.
\begin{lemma}\label{lem:m-concentration-perc}
	$M(\d_A) = p^2 M (1 + o(1))$. With probability $1 - o(1)$, $M(\d_S) \sim M(\d_A)$ and $M_2(\d_S) \sim M_2(\d_A)$.
\end{lemma}
\begin{proof}[Proof sketch]
	The proof of this claim is analogous to \Cref{lem:m-concentration}, using \Cref{lem:degree-sequence-concentration-perc} instead of \Cref{lem:dh-ds-comparison}. The same proof method works, noting that $(2\delta^{1/32} \gamma \cop)^2 \frac{p\cop^6}{\sqrt{\log M}} = o(p^2M /\cop)$. 
\end{proof}

\begin{remark}\label{rem:degree-0-perc}
		Recall that $\deg^*$ is the sequence $\deg$ with all entries equal to $0$ removed. An analogous argument to \Cref{lem:degree-0-vertices}, calling on \Cref{lem:degree-sequence-concentration-perc} instead of \Cref{lem:dh-ds-comparison}, implies that $n(\d_A^*) = \Omega(p^2n)$ and \aas\ $n(\d_A^*) \sim n(\d_S^*)$. 
\end{remark}

\section{Giant components in the percolated graph}

	The proof of \Cref{thm:perc-gc-nice} follows quickly from the previously proved results about the percolated random graph model and arguments analogous to those used to prove \Cref{thm:main-gc-pretty}. 

	\begin{lemma}\label{lem:mj-conc-perc}
		$\MJ(\d_A) = \Theta(M(\d_A))$, and \aas\ $\MJ(\d_A) - \MJ(\d_S) = o(M(\d_A))$. 
	\end{lemma}
	\begin{proof}
		The proof of this lemma is essentially the same as \Cref{lem:joos-m-concentration}, applying \Cref{lem:m-concentration-perc} instead of \Cref{lem:m-concentration} and \Cref{lem:degree-sequence-concentration-perc} instead of \Cref{lem:dh-ds-comparison}, so we omit the details.
	\end{proof}
	
	\begin{lemma}\label{lem:rd-conc-good}
		If $S$ is good, then \aas\ $|R(\d_S(S)) - R(\d_A)| = o(M(\d_A))$.
	\end{lemma}
	\begin{proof}
		The proof of this claim is very similar to the proof of \Cref{lem:rd-concentration}, so we omit the details. The notable differences are that we apply \Cref{lem:m-concentration-perc} instead of \Cref{lem:m-concentration}, and that the numbers of entries in each sequence $\d_S$ or $\d_A$ with value $0$, $1$, or $2$ differ by at most $o\left( \frac{\gamma\cop^6}{\sqrt{\log M}} \right) + \sum_{i=0}^2 \frac{\gamma n_i(\d_H)}{\co^{3}}$.
	\end{proof}
	
	Similar to the case where $S$ is fixed, we define $\Dsp$ to be the set of all possible sequences of all induced subgraphs of every $G \in \Gnd$. We then define $\gooddsp$ to be the subset of these sequences that satisfy the bounds stated in parts (a)--(d) of \Cref{lem:degree-sequence-concentration-perc}. Again, the definition of $\gooddsp$ is implicitly dependent on $\delta$, but we omit this.
	
	\begin{proof}[Proof of \Cref{thm:perc-gc-nice}]
		\Cref{lem:degree-sequence-concentration-perc} implies that $n(\d_S) = |S| = np \pm 3 \sqrt{np\log n}$ for all $\d_S \in \gooddsp$. \Cref{rem:degree-0-perc} then implies that 
		\begin{align*}
			n(\d_S^*) \sim n(\d_H^*) = \lfloor np \rfloor - n_0(\d_H) \pm 3 \sqrt{np\log n}.
		\end{align*}
		Since $np - n_0(\d_H) \geq p^2 n$ (by \Cref{rem:degree-0-perc}), this implies that $n(\d_S^*) = (np - n_0(\d_H))(1 + o(1))$. Then
		\Cref{lem:degree-sequence-concentration-perc} and \Cref{lem:m-concentration-perc,lem:rd-conc-good} imply that if $R(\d_A) \geq \varepsilon \MJ(\d_A)$ for some $\varepsilon > 0$, then \aas\ $R(\d_S) \geq \frac{1}{2}\varepsilon \MJ(\d_S)$, and conversely that if $R(\d_A) \leq \delta' \MJ(\d_A)$ for some $\delta'\to 0$, then \aas\ $R(\d_S) \leq 2\delta'\MJ(\d_A)$. \Cref{lem:m-concentration-perc} also implies that $\d_A$ is well-behaved and \aas\ $\d_S$ is well-behaved. Then applying the law of total probability and \Cref{thm:joos} to each possible choice of $\d_S$ completes the proof. 
	\end{proof}

\end{document}